\documentclass[a4paper,12pt,intlimits,oneside]{amsart}

\usepackage{latexsym,mathrsfs}
\usepackage{amsbsy,bm}

\usepackage{amsmath}
\usepackage{amssymb}
\usepackage{amsfonts}
\usepackage{amsthm,amscd}

\newtheorem{thm}{Theorem}[section]
\newtheorem{prop}[thm]{Proposition}
\newtheorem{cor}[thm]{Corollary}
\newtheorem{lem}[thm]{Lemma}

\theoremstyle{definition}
\newcommand{\comment}[1]{}

\numberwithin{equation}{section}

\newcommand{\epf}{ $\Box$\medskip}

\def\lsim{\raisebox{-1ex}{$~\stackrel{\textstyle <}{\sim}~$}}

\theoremstyle{definition}

\begin{document}
\title[Representation theorem and Hardy-Orlicz spaces]{A generalization of representation theorems in Hardy-Orlicz spaces on the upper complex half-plane}
\author[J. M. T. Dje]{Jean-Marcel Tanoh Dje}
\address{Unité de recherche et d'expertises numériques, Université virtuelle de Côte d'Ivoire, Cocody II Plateaux - 28BP536 Abidjan 28}
\email{{\tt tanoh.dje@uvci.edu.ci}}
\author[J. Feuto]{Justin Feuto}
\address{Laboratoire de Math\'ematiques et Application, UFR Math\'ematiques et Informatique, Universit\'e F\'elix Houphou\"et-Boigny Abidjan-Cocody, 22 B.P 1194 Abidjan 22. C\^ote d'Ivoire}
\email{{\tt feuto.justin@ufhb.edu.ci}}


\date{}
\begin{abstract}
In this paper, we give Poisson and Cauchy representation theorems in
 Hardy-Orlicz spaces on the upper complex half-plane. We use these theorems
 for the construction of  dual spaces of certain Hardy-Orlicz spaces and also for the characterization of some classical operators in Orlicz spaces.
\end{abstract}



\maketitle

\section{Introduction.}
Representation theorems are very important tools when working in Hardy spaces. 
Regarding Hardy space on the  upper complex half-plane, there are two versions of this  theorem depending on whether the space is defined as a subspace of harmonic spaces or holomorphic spaces.

Recall that the  upper complex half-plane is the subset $\mathbb{C}_{+}$ of the complex plane $\mathbb{C}$ defined by
$$\mathbb{C}_{+}:=\left\{ z=x+iy \in \mathbb{C} :   y > 0 \right\}. $$
We denote by  $h^{p}(\mathbb{C}_{+})$  the
space of all harmonic functions $F$ on $\mathbb{C}_{+}$ wich satisfy
$$ \|F\|_{h^{p}}:=\sup_{y > 0}\|F(.+iy)\|_{L^{p}} < \infty,     $$
where $\|\cdot\|_{L^{p}}$ stands for the usual Lebesgue norm on $\mathbb R$.
For $p\geq 1$,  $(h^{p}(\mathbb{C}_{+}), \|.\|_{h^{p}})$ is a Banach space.

The Poisson representation theorem  given in \cite[Theorem 11.6]{javadmas} can be reformulated as follow.

Let $1<p<\infty$ and $F$ be a harmonic function on $\mathbb{C}_{+}$. The following assertions are equivalent:
\begin{itemize}
\item[(i)]  $F\in h^{p}(\mathbb{C}_{+})$. 
\item[(ii)] There exists a unique $f\in L^{p}(\mathbb{R})$ such that
\begin{equation}\label{eq:luxboqurg}
F(x+iy)= \frac{1}{\pi}\int_{\mathbb{R}}\frac{y}{(x-t)^{2}+y^{2}}f(t)dt,~~ \forall~x+iy \in \mathbb{C_{+}}.\end{equation}
Moreover,  
$$  \|F\|_{h^{p}}  =\|f\|_{L^{p}}.    $$
\end{itemize}
The above result coupled with the fact that the Poisson integral (\ref{eq:luxboqurg}) is an isometric isomorphism of $h^{p}(\mathbb{C}_{+})$ onto $L^{p}(\mathbb{R})$ allow  to identify the Hardy space $h^{p}(\mathbb{C}_{+})$ with the Lebesgue space $L^{p}(\mathbb{R})$. 

Denote by $\mathcal{M}(\mathbb{R})$ the set of  bounded Radon measures on  $\mathbb{R}$. 
Equipped with the norm
$$ \mu\longmapsto\|\mu\|_{\mathcal{M}}:=|\mu|(\mathbb{R}),      $$
$\mathcal{M}(\mathbb{R})$ is a Banach space. Notice that $|\mu|$ is the total variation of $\mu$.\\
For the case $p=1$, we use Radon measures for the representation result.
For a harmonic function  $F$ on $\mathbb{C}_{+}$, the following assertions are equivalent:
\begin{itemize}
 \item[(i)]  $F\in h^{1}(\mathbb{C}_{+})$. 
 \item[(ii)] There exists a unique $\mu \in \mathcal{M}(\mathbb{R})$ such that
 $$ F(x+iy)= \frac{1}{\pi}\int_{\mathbb{R}}\frac{y}{(x-t)^{2}+y^{2}}d\mu(t),~~ \forall~x+iy \in \mathbb{C_{+}}.  $$
Moreover,  
 $$\|F\|_{h^{1}}  =\|\mu\|_{\mathcal{M}}.$$
 \end{itemize}
As in the case of $p>1$,  we can 
identify $h^{1}(\mathbb{C}_{+})$ with $\mathcal{M}(\mathbb{R})$.

We know consider the spaces 
$H^{p}(\mathbb{C}_{+})$ consists of all analytic functions $F$ on $\mathbb{C}_{+}$ satisfying
$$ \|F\|_{H^{p}}:=\sup_{y > 0}\|F(.+iy)\|_{L^{p}} < \infty     $$
and 
$$  H^{p}(\mathbb{R})=\left\{ f\in L^{p}(\mathbb{R}): \int_{\mathbb{R}}\frac{f(t)}{t-\overline{z}}dt=0, ~~\forall~z \in \mathbb{C}_{+}   \right\}.
      $$
For $p\geq 1$, $(H^{p}(\mathbb{R}), \|.\|_{L^{p}})$ (resp. $(H^{p}(\mathbb{C}_{+}), \|.\|_{H^{p}})$) is a closed subspace of  $(L^{p}(\mathbb{R}), \|.\|_{L^{p}})$
 (resp. $(h^{p}(\mathbb{C}_{+}), \|.\|_{h^{p}})$).

Cauchy's representation theorem in Hardy spaces (see for example \cite{Jbgarnett},\cite{javadmas}) can be reformulated as follows: let $1\leq p<\infty$ and $F$ be an analytic function on $\mathbb{C}_{+}$. The following assertions are equivalent:
\begin{itemize}
\item[(i)]  $F\in H^{p}(\mathbb{C}_{+})$. 
\item[(ii)] There exists a unique $f\in H^{p}(\mathbb{R})$ such that
\begin{equation}\label{eq:luxboqaurg}
F(z)= \frac{1}{2\pi i}\int_{\mathbb{R}}\frac{f(t)}{t-z}dt,~~ \forall~z \in \mathbb{C_{+}}.\end{equation}
Moreover,  
$$  \|F\|_{H^{p}}  =\|f\|_{L^{p}}.      $$
\end{itemize}
The Cauchy integral (\ref{eq:luxboqaurg}) is an isometric isomorphism of the Hardy space $H^{p}(\mathbb{C}_{+})$ onto the space $H^{p}(\mathbb{R})$. Thus, the Hardy space $H^{p}(\mathbb{C}_{+})$ is identified with a closed subspace of $L^{p}(\mathbb{R})$.

Note that the use of the Poisson representation theorem (resp. Cauchy) has made it possible to simplify the demonstration of several problems related to Hardy spaces. Let us quote among others, the control of certain classical operators in Hardy spaces (see for example \cite{berosim},\cite{Pduren2},\cite{Jbgarnett},\cite{koossi},\cite{javadmas}) and the characterization of  dual of Hardy's space (see \cite{Jbgarnett}).

In view of what precedes, it seemed useful to extend these results to Orlicz's space which generalizes  Lebesgue's space.

This paper is organized as follows :

In the next section, we state our main results.
Section III is devoted to some useful properties and prerequisites on Hardy-Orlicz spaces. In section IV, we give the proof of the results in section II.

We propose the following abbreviation $\mathrm{ A}\lsim \mathrm{ B}$ for the inequalities $\mathrm{ A}\leq C\mathrm{ B}$, where $C$ is a positive constant independent of the main parameters. If $\mathrm{ A}\lsim \mathrm{ B}$ and $\mathrm{ B}\lsim \mathrm{ A}$, then we write $\mathrm{ A}\approx \mathrm{ B}$.
In all what follows, the letter $C$ will be used for non-negative constants independent of the relevant variables that may change from one occurrence to another. Constants with subscript, such as $C_{s}$, may also change in different occurrences, but depend on the parameters mentioned in it. 

\section{Statement of main results.}

In this paper, a continuous and nondecreasing function $\Phi$ from $\mathbb{R}_{+}$ onto itself is called a growth function. Observe that if $\Phi$ is a growth function, then $\Phi(0)=0$ and $\lim_{t \to +\infty}\Phi(t)=+\infty$. If  $\Phi(t)> 0$ for all  $t> 0$ then $\Phi$ is a homeomorphism of $\mathbb{R}_{+}$ onto $\mathbb{R}_{+}$.
 
Let $\Phi$ be a  growth function. 
\begin{itemize}
\item The Orlicz space  $L^{\Phi}(\mathbb{R})$ is the set of all equivalent
classes (in the usual sense) of measurable functions $f : \mathbb{R} \longrightarrow \mathbb{C}$ which satisfy
$$ \|f\|_{L^{\Phi}}^{lux}:=\inf\left\{\lambda>0 : \int_{\mathbb{R}}\Phi\left(\dfrac{| f(x)|}{\lambda}\right)dx \leq 1  \right\}< \infty. $$
\item The Hardy$-$Orlicz space  $h^{\Phi}(\mathbb{C_{+}})$ is the set of harmonic functions on $\mathbb{C_{+}}$ which satisfy 
$$   \|F\|_{h^{\Phi}}^{lux}:=\sup_{y> 0}\|F(.+y)\|_{L^{\Phi}}^{lux} < \infty.   $$  
\end{itemize}
If $\Phi$ is convex (resp. convex and  $\Phi(t)> 0$ for all  $t> 0$) then $(L^{\Phi}(\mathbb{R}),  \|.\|_{L^{\Phi}}^{lux})$ (resp. $\left(h^{\Phi}(\mathbb{C_{+}}), \|.\|_{h^{\Phi}}^{lux}\right)$) is a Banach space.  The space  $L^{\Phi}$ (resp.  $h^{\Phi}$) generalizes the Lebesgue space $L^{p}$ (resp. Hardy space $h^{p}$) for $0< p < \infty$.

Let $\Phi$ be a convex growth function. We say that $\Phi$ is an $N-$function if the following assertions are
satisfied:
\begin{itemize}
\item[(i)] $\Phi(t)=0$ if and only if $t=0$,
\item[(ii)] $\lim_{t \to 0}\frac{\Phi(t)}{t}=0$  \hspace*{0.25cm}\text{and} \hspace*{0.25cm} $\lim_{t \to \infty}\frac{\Phi(t)}{t}=+\infty.$
\end{itemize}
The map $t\mapsto \exp(t)-t-1$ is an example of $N-$function.

\begin{thm}\label{pro:main11a6}
Let $\Phi$ be an N$-$function and $F$ a harmonic function on $\mathbb{C}_{+}$. The following assertions are equivalent:
\begin{itemize}
\item[(i)]  $F\in h^{\Phi}(\mathbb{C}_{+})$.
\item[(ii)] There exists a unique  $f \in L^{\Phi}(\mathbb{R})$ such that 
 \begin{equation}\label{eq:556655656}
 F(x+iy)= \frac{1}{\pi}\int_{\mathbb{R}}\frac{y}{(x-t)^{2}+y^{2}}f(t)dt,~~ \forall~x+iy \in \mathbb{C_{+}}.    \end{equation}
\end{itemize}
Moreover, we have
$$  \|F\|_{h^{\Phi}}^{lux}=\lim_{y \to 0}\|F(.+iy)\|_{L^{\Phi}}^{lux}=\| f\|_{L^{\Phi}}^{lux}.         $$
\end{thm}

If  $\Phi(t)=t^{p}$ with $1< p< \infty$ then Theorem \ref{pro:main11a6} corresponds to  \cite[Theorem 11.6]{javadmas}. If  $\Phi$ is an N$-$function then the Poisson integral (\ref{eq:556655656}) is an isometric isomorphism
 of $h^{\Phi}(\mathbb{C}_{+})$ onto $L^{\Phi}(\mathbb{R})$.

\begin{thm}\label{pro:main11aaq6}
Let $\Phi$ be a convex growth function such that $\Phi(t)> 0$ for all  $t> 0$. For  $F\in h^{\Phi}(\mathbb{C}_{+})$, there exists a unique Borel measure $\mu$ such that
$$ \int_{\mathbb{R}}\frac{d|\mu|(t)}{1+t^{2}} < \infty 
  $$
and
$$   F(x+iy)= \frac{1}{\pi}\int_{\mathbb{R}}\frac{y}{(x-t)^{2}+y^{2}}d\mu(t),~~ \forall~x+iy \in \mathbb{C_{+}}.      $$
Moreover,
\begin{equation}\label{eq:deltaqa2}
\lim_{y\to 0}\int_{\mathbb{R}}\varphi(x) F(x+iy)dx= \int_{\mathbb{R}}\varphi(x)d\mu(x), ~~ \forall~\varphi\in C_{c}(\mathbb{R}),\end{equation}
where $C_{c}(\mathbb{R})$ is the set of continuous functions with compact support on $\mathbb{R}$.
\end{thm}

Let $\Phi$ be a growth function. We say that
\begin{itemize}
\item $\Phi$ satisfies the $\Delta_{2}-$condition (or $\Phi \in \Delta_{2}$) if there exists a constant
$K > 1$ such that
\begin{equation}\label{eq:delta2}
\Phi(2t) \leq K \Phi(t),~ \forall~ t > 0.\end{equation}
\item  $\Phi$ is of upper-type $q>0$ if there exists a constant
$C_{q}> 0$ such that
\begin{equation}\label{eq:sui8n}
\Phi(st)\leq C_{q}t^{q}\Phi(s),~~\forall~s>0, ~\forall~t\geq 1.\end{equation}
\end{itemize}
For $q \geq 1$, we denote by $\mathscr{U}^{q}$ the set of all growth functions of upper-type $q$ such that the function 
$t\mapsto \frac{\Phi(t)}{t}$ is non decreasing on $\mathbb{R}_{+}^{*}= \mathbb{R}_{+}\backslash\{0\}$,  and we put $\mathscr{U}:=\bigcup_{q\geq 1}\mathscr{U}^{q}.$
It is obvious that any growth function  $\Phi \in \mathscr{U}$ satisfies the  $\Delta_{2}-$condition. Moreover, any element in $\mathscr{U}$ is a homeomorphism of $\mathbb{R}_{+}$ onto $\mathbb{R}_{+}$. It is obvious that 
the function $t \mapsto \frac{t^{q}}{\ln(e+t)}$ belongs to $\mathscr{U}^{q}$ for  $q\geq 1$.

We say that two growth functions $\Phi_{1}$ and $\Phi_{2}$ are equivalent, if there exists a constant $c > 0$ such that

\begin{equation}\label{eq:equivalent}
c^{-1}\Phi_{1}(c^{-1}t) \leq \Phi_{2}(t)\leq c\Phi_{1}(ct), ~~ \forall~ t > 0.\end{equation} 
We will assume in the sequel that  any element of $\mathscr{U}$ is  
convex and  belongs to $\mathscr{C}^{1}(\mathbb{R}_{+})$ (see for example \cite{jmtafeuseh, djesehb, djesehaqb, sehbaedgc}).

Let $\Phi$ be a growth function. We note  $H^{\Phi}(\mathbb{R})$  the set of functions $f\in L^{\Phi}(\mathbb{R})$ such that
$$  \int_{\mathbb{R}}\frac{f(t)}{t-\overline{z}}dt=0, ~~\forall~z \in \mathbb{C}_{+},   $$
and  $H^{\Phi}(\mathbb{C_{+}})$  the subspace of  $h^{\Phi}(\mathbb{C_{+}})$ consisting of analytic functions on $\mathbb{C_{+}}$.\\
If  $\Phi$ is convex and $\Phi(t)> 0$ for all  $t> 0$ (resp.  $\Phi \in \mathscr{U}$) then  $\left(H^{\Phi}(\mathbb{C_{+}}), \|.\|_{H^{\Phi}}^{lux}\right)$ (resp.  $\left(H^{\Phi}(\mathbb{R}), \|.\|_{L^{\Phi}}^{lux}\right)$) is a closed subspace of  $\left(h^{\Phi}(\mathbb{C_{+}}), \|.\|_{h^{\Phi}}^{lux}\right)$ (resp.  $\left(L^{\Phi}(\mathbb{R}), \|.\|_{L^{\Phi}}^{lux}\right)$), where  $\|.\|_{H^{\Phi}}^{lux}$ is the restriction of $\|.\|_{h^{\Phi}}^{lux}$ on $H^{\Phi}(\mathbb{C_{+}})$.

\begin{thm}\label{pro:main11aqa6}
Let  $\Phi \in \mathscr{U}$ and $F$ be an analytic function on $\mathbb{C}_{+}$. The following assertions are equivalent:
\begin{itemize}
\item[(i)]  $F\in H^{\Phi}(\mathbb{C}_{+})$.
\item[(ii)] There exists   $f \in H^{\Phi}(\mathbb{R})$ such that 
$$  F(x+iy)=\frac{1}{\pi}\int_{\mathbb{R}}\frac{y}{(x-t)^{2}+y^{2}}f(t)dt, ~~\forall~x+iy \in \mathbb{C}_{+}.
    $$
\end{itemize}
If so, $f$ is unique and we also have for all $z=x+iy \in \mathbb{C}_{+}$
$$ F(z)=\frac{i}{\pi}\int_{\mathbb{R}}\frac{x-t}{(x-t)^{2}+y^{2}}f(t)dt
=\frac{1}{2\pi i}\int_{\mathbb{R}}\frac{f(t)}{t-z}dt.
  $$
Moreover, 
$$  \lim_{y \to 0}\|F(.+iy)-f\|_{L^{\Phi}}^{lux}=0   $$
and
$$  \|F\|_{H^{\Phi}}^{lux}=\lim_{y \to 0}\|F(.+iy)\|_{L^{\Phi}}^{lux}=\| f\|_{L^{\Phi}}^{lux}.        $$
\end{thm}

If  $\Phi(t)=t^{p}$ with $1\leq p< \infty$ then Theorem \ref{pro:main11aqa6} corresponds to \cite[Theorem 13.2]{javadmas}. If  $\Phi \in \mathscr{U}$ then the Cauchy integral is an isometric isomorphism between $H^{\Phi}(\mathbb{C}_{+})$ and $H^{\Phi}(\mathbb{R})$. 

We presented in the above results the theorems of representation in Hardy-Orlicz spaces. In the rest of our presentation, we state some applications of these theorems in certain theories related to Hardy spaces.

Let $\Phi$ be a convex growth function.  The complementary function of $\Phi$ is the function $\Psi$ defined by
$$ \Psi(s)=\sup_{t\geq 0}\{st-\Phi(t) \}, ~ \forall~  s \geq 0.       $$
Let $\Phi$ be a convex growth function and $\Psi$ its complementary. In general, $\Psi$ is not a growth function. But, if $\Phi$ is an $N-$function then $\Psi$ is also an $N-$function and the complementary of $\Psi$ is $\Phi$ (see for example \cite{shuchen, pkmlusg, raoren, rao68ren}).

Let $\Phi$ be an $N-$function. We say that $\Phi$ satisfies $\nabla_{2}-$condition (or $\Phi \in \nabla_{2}$)  if $\Phi$ and its complementary function both satisfy $\Delta_{2}-$condition.\\
Let $\alpha > 1$ and $\beta> 0$. The function $t\mapsto t^{\alpha}[\ln(1+t)]^{\beta}$ is an example of an $N-$function satisfying the $\nabla_{2}-$condition.

In the sequel, we will call complementary pair of $N-$functions the pair $(\Phi, \Psi)$ in which $\Phi$ is the  complementary function of $\Psi$ and vice versa.
We will also assume that the N-functions are of class $\mathscr{C}^{1}(\mathbb{R}_{+})$.

Let $f$ be a measurable function on $\mathbb{R}$. The Hilbert transform $\mathcal{H}(f)$ and the Cauchy integral  $S(f)$  of $f$ are respectively defined by
$$   \mathcal{H}(f)(x):=\lim_{\varepsilon \to 0}\frac{1}{\pi}\int_{|x-t|>\varepsilon}\frac{f(t)}{x-t}dt , ~~\forall~ x \in \mathbb{R}   $$
and 
$$  S(f)(z)= \frac{1}{i\pi }\int_{\mathbb{R}}\frac{f(t)}{t-z}dt,~~ \forall~z \in \mathbb{C_{+}}. 
  $$

\begin{thm}\label{pro:main11pal6}
Let $\Phi$ be an $N-$function and  $f \in L^{\Phi}(\mathbb{R})$ with real values. If $\Phi \in\nabla_{2}$ then  the following equivalent assertions are satisfied. 
\begin{itemize}
\item[(i)] $S(f) \in H^{\Phi}(\mathbb{C}_{+})$, 
\item[(ii)] $(f+i \mathcal{H}(f)) \in H^{\Phi}(\mathbb{R})$.
\end{itemize}
Moreover, there is a constant $C:=C_{\Phi}>0$ such that
\begin{equation}\label{eq:5qaq656}
\| f\|_{L^{\Phi}}^{lux}\leq\|S(f)\|_{H^{\Phi}}^{lux}\leq C\| f\|_{L^{\Phi}}^{lux}.    \end{equation}
\end{thm}

Let $f$ be a measurable function on $\mathbb{R}$. The Hilbert maximal function $\mathcal{\widetilde{H}}(f)$ of a function $f$ is defined by
$$  \mathcal{\widetilde{H}}(f)(x):=\sup_{\varepsilon>0}\left|\frac{1}{\pi}\int_{|x-t|>\varepsilon}\frac{f(t)}{x-t}dt \right| , ~~\forall~ x \in \mathbb{R}.   $$

Let $\alpha>0$ and $F$ be a harmonic function on $\mathbb{C_{+}}$. The radial maximal  function $\mathcal{M}_{rad}(F)$ and the non-tangential  maximal  function  $\mathcal{M}_{ntg}^{\alpha}(F)$ of  $F$ are respectively defined by 
$$ \mathcal{M}_{rad}(F)(t):=\sup_{y>0}|F(t+iy)|, ~~\forall~ t \in \mathbb{R}   $$
and 
$$ \mathcal{M}_{ntg}^{\alpha}(F)(t):=\sup_{z\in \Gamma_{\alpha}(t)}|F(z)|, ~~\forall~ t \in \mathbb{R},   $$
where  $ \Gamma_{\alpha}(t) := \left\{ x+iy \in \mathbb{C_{+}}: | x-t | < \alpha y  \right\}.  $ 

\begin{thm}\label{pro:main 5dmw3pl}
  Let $\Phi$ be an  N-function such that $\Phi \in \Delta_{2}$. The following assertions are equivalent.
\begin{itemize}
 \item[(i)] $\Phi\in \nabla_{2}$.
 \item[(ii)] There exists a constant $C_{1}>0$ such that for all $F\in h^{\Phi}(\mathbb{C}_{+})$,   
   \begin{equation}\label{eq:cjlis}
   \|F\|_{h^{\Phi}}^{lux}\leq\|\mathcal{M}_{rad}(F)\|_{L^{\Phi}}^{lux}\leq\|\mathcal{M}_{ntg}^{\alpha}(F)\|_{L^{\Phi}}^{lux} \leq C\|F\|_{h^{\Phi}}^{lux}.
   \end{equation}
 \item[(iii)] There exists a constant $C_{2}>0$ such that for all $f\in L^{\Phi}(\mathbb{R})$,   
  \begin{equation}\label{eq:cotjplaqlis}
 \|f\|_{L^{\Phi}}^{lux}\leq \|\mathcal{H}(f)\|_{L^{\Phi}}^{lux} \leq \|\mathcal{\widetilde{H}}(f)\|_{L^{\Phi}}^{lux} \leq C_{2}\|f\|_{L^{\Phi}}^{lux}.
  \end{equation}
 \end{itemize}
 \end{thm}

Let $(\Phi,\Psi)$ be a complementary pair of $N-$functions. If  both $\Phi$ and $\Psi$  satisfy  $\Delta_2-$condition  then the map 
$$  F\mapsto \|F\|_{h^{\Phi}}^{0}:= \sup\left\{  \lim_{y\to 0}\int_{\mathbb{R}}|F(x+iy)G(x+iy)|dx    :  G\in h^{\Psi}(\mathbb{C}_{+})\hspace*{0.15cm} \text{with}\hspace*{0.15cm} \|G\|_{h^{\Psi}}^{lux} \leq 1 \right\},
  $$
is a norm on  $h^{\Phi}(\mathbb{C}_{+})$. Moreover, 
\begin{equation}\label{eqwpn}
\|F\|_{h^{\Phi}}^{lux}\leq \|F\|_{h^{\Phi}}^{0} \leq 2\|F\|_{h^{\Phi}}^{lux}, ~~ \forall~ F\in h^{\Phi}(\mathbb{C}_{+}).
\end{equation} 

\begin{thm}\label{pro:main 5aqk3pl}
Let $(\Phi,\Psi)$ be a complementary pair of $N-$functions. If  $\Phi \in  \nabla_{2}$ then the following assertions are satisfied.
\begin{itemize}
 \item[(i)] The topological dual of $h^{\Phi}(\mathbb{C}_{+})$, $\left(h^{\Phi}(\mathbb{C}_{+})\right)^{*}$ is isometrically isomorphic to $h^{\Psi}(\mathbb{C}_{+})$, in the sense that, for  $T\in \left(h^{\Phi}(\mathbb{C}_{+})\right)^{*}$ there is a unique $G\in h^{\Psi}(\mathbb{C}_{+})$ such that  
$$ T(F)= \lim_{y\to 0}\int_{\mathbb{R}}F(x+iy)G(x+iy)dx, ~~ \forall~F\in h^{\Phi}(\mathbb{C}_{+})
    $$
and 
$$  \|T\|_{(h^{\Phi})^{*}} = \|G\|_{h^{\Psi}}^{0}.     $$
 \item[(ii)]  The topological dual of $H^{\Phi}(\mathbb{C}_{+})$,   $\left(H^{\Phi}(\mathbb{C}_{+})\right)^{*}$    is isomorphic to $H^{\Psi}(\mathbb{C}_{+})$, in the sense that, for all  $T\in \left(H^{\Phi}(\mathbb{C}_{+})\right)^{*}$, there is a unique $G\in H^{\Psi}(\mathbb{C}_{+})$ such that
 $$   T(F)= \lim_{y\to 0}\int_{\mathbb{R}}F(x+iy)\overline{G(x+iy)}dx, ~~ \forall~F\in H^{\Phi}(\mathbb{C}_{+}). 
   $$
\end{itemize}
\end{thm}

\section{Some definitions and useful properties}

We present in this section some useful results needed in our presentation.

\subsection{Some properties of Orlicz space on $\mathbb{R}$.} 

Let  $\Phi \in \mathscr{C}^{1}(\mathbb{R}_{+})$  a growth function. The lower and the upper indices of $\Phi$ are respectively defined by
$$ a_\Phi:=\inf_{t>0}\frac{t\Phi'(t)}{\Phi(t)}
  \hspace*{1cm}\textrm{and} \hspace*{1cm} b_\Phi:=\sup_{t>0}\frac{t\Phi'(t)}{\Phi(t)}.          $$
Let $\Phi \in \mathscr{C}^{1}(\mathbb{R}_{+})$ a growth function. The following assertions are satisfied.
\begin{itemize}
\item[(i)] If $\Phi$ is convex then $\Phi \in \Delta_{2}$ if and only if  $\Phi \in \mathscr{U}$      if and only if  $1\leq a_\Phi\leq b_\Phi <\infty$ (see. \cite[Page 171]{djesehb}).
\item[(ii)] If  $\Phi$ is an $N-$function  then  $\Phi \in \nabla_{2}$ if and only if  $1< a_\Phi\leq b_\Phi <\infty$ (see. \cite[Lemma 2.1]{denhuan}). Moreover, if  $\Phi \in \mathscr{U}$ then  $\Phi \in \nabla_{2}$ if and only if there is a constant  $C:=C_{\Phi}>0$ such that for all $t > 0$,
\begin{equation}\label{eq:saaui8n}
\int_{0}^{t}\frac{\Phi(s)}{s^{2}}ds \leq C \frac{\Phi(t)}{t},
\end{equation}
(see. \cite[Lemma 3.1]{djesehb}).
\item[(iii)] If  $0< a_\Phi\leq b_\Phi <\infty$ then the function  $t\mapsto \frac{\Phi(t)}{t^{a_\Phi}}$ is increasing on $\mathbb{R}_{+}^{*}$ while the function  $t\mapsto \frac{\Phi(t)}{t^{b_\Phi}}$ is decreasing on $\mathbb{R}_{+}^{*}$ (see. \cite[Lemma 2.1]{sehbaedgc}).
\end{itemize}

Let  $p>0$ and $\Phi$ be a growth function. We say that $\Phi$ is of lower-type $p$ if there exists a constant
$C_{p}> 0$ such that
\begin{equation}\label{eq:saui8n}
\Phi(st)\leq C_{p}t^{p}\Phi(s),~~\forall~s>0, ~\forall~0< t\leq 1.\end{equation}
For $0< p\leq 1$, we denote by $\mathscr{L}^{p}$ the set of all growth functions of lower-type $p$  such that the function 
$t\mapsto \frac{\Phi(t)}{t}$ is non-increasing on $\mathbb{R}_{+}^{*}$, and  we put $\mathscr{L}:=\bigcup_{0< p\leq 1}\mathscr{L}^{p}.$
If $\Phi$ is a one-to-one growth function then $\Phi \in  \mathscr{U}^{q}$ if and only if $\Phi^{-1} \in  \mathscr{L}^{1/q}$ (see. \cite[Proposition 2.1]{sehbatchoundja1}).

Let $\Phi$ be a  growth function.  The Morse$-$Transue space  $M^{\Phi}(\mathbb{R})$ is the set of all equivalent
classes (in the usual sense) of measurable functions $f : \mathbb{R} \longrightarrow \mathbb{C}$ such that
$$  \int_{\mathbb{R}}\Phi\left(\frac{|f(x)|}{\lambda}\right)dx < \infty, ~~\text{for all}~ \lambda>0   .        $$

Most of the results below can be found in the following references (\cite{shuchen},\cite{pkmlusg},\cite{raoren}).

Let $\Phi$ be a convex growth function. The following assertions are satisfied.

\begin{itemize}
  \item[(i)] $(M^{\Phi}(\mathbb{R}),  \|.\|_{L^{\Phi}}^{lux})$ is a closed subspace of  $(L^{\Phi}(\mathbb{R}),  \|.\|_{L^{\Phi}}^{lux})$.
  \item[(ii)] If $\Phi(t)>0$,  for all $t>0$ then $M^{\Phi}(\mathbb{R})$ is separable.
 \item[(iii)] If  $\Phi \in \Delta_{2}$ then $M^{\Phi}(\mathbb{R})=L^{\Phi}(\mathbb{R})$.
\end{itemize}

 The following result can be found for example in \cite{bansahsehba}.
 
\begin{lem}\label{pro:main26}
 Let $\alpha$ be a real. Then for $y > 0$ fixed, the integral 
$$ \textit{J}_{\alpha}(y) =\int_{\mathbb{R}}\dfrac{dx}{|x+iy|^{\alpha}},    $$ 
converges if and only if $\alpha > 1$. In this case,
$$ \textit{J}_{\alpha}(y)=y^{1-\alpha}\int_{0}^{\infty}\dfrac{du}{u^{1/2}(1+u)^{\alpha/2}}.     $$ 
 \end{lem} 

For $y>0$, the Poisson kernel $P_{y}$ and its conjugate   $Q_{y}$ are the functions defined respectively by
$$   P_y(x)=\dfrac{1}{\pi}\dfrac{y}{x^{2}+y^{2}}, ~\forall~ x \in \mathbb{R},    $$
and
$$   Q_y(x)=\dfrac{1}{\pi}\dfrac{x}{x^{2}+y^{2}}, ~\forall~ x \in \mathbb{R}.    $$

 The following result is an immediate consequence of Lemma \ref{pro:main26}. Therefore, the proof will be omitted.

\begin{prop}\label{pro:main70aq6}
Let   $y >0$ and $\Phi$ a convex growth function. The following assertions are satisfied.
\begin{itemize}
\item[(i)]  $P_y \in M^{\Phi}(\mathbb{R})$.
\item[(ii)] If  $\Phi$ is  $\mathscr{C}^{1}(\mathbb{R}_{+})$ and  $1< a_\Phi\leq b_\Phi <\infty$ then   $Q_y \in L^{\Phi}(\mathbb{R})$.
\end{itemize}
\end{prop}

Let $\Phi$ be a growth function. We have the following inclusions. 
\begin{itemize}
\item[(a)] If  $\Phi$ is convex  then 
$$   C_{c}(\mathbb{R}) \subseteq M^{\Phi}(\mathbb{R}) \subseteq  L^{\Phi}(\mathbb{R}) \subseteq L^{1}\left(\frac{dt}{1+t^{2}}\right).
        $$ 
\item[(b)] If $\Phi \in \mathscr{C}^{1}(\mathbb{R}_{+})$ and  $1\leq a_\Phi \leq b_\Phi<\infty$ then 
$$ L^{a_\Phi}(\mathbb{R}) \cap L^{b_\Phi}(\mathbb{R}) \subseteq L^{\Phi}(\mathbb{R}) \subseteq L^{a_\Phi}(\mathbb{R}) + L^{b_\Phi}(\mathbb{R}) \subseteq L^{1}\left(\frac{dt}{1+|t|}\right) \cap L^{1}\left(\frac{dt}{1+t^{2}}\right).
       $$
\end{itemize}

\begin{prop}\label{pro:main706}
Let $\Phi$ be a convex growth function. If  $\Phi\in \Delta_{2}$ then $C_{c}(\mathbb{R})$ is dense in $L^{\Phi}(\mathbb{R})$.
\end{prop} 

\begin{proof}
Let $f \in L^{\Phi}(\mathbb{R})$  and  $\varepsilon >0$. 
Since the set of simple functions on $\mathbb{R}$ is dense in $L^{\Phi}(\mathbb{R})$ (see.  \cite[Corollary 5]{raoren}),  there exists $g$ a simple function on $\mathbb{R}$ (i.e; $g=\sum_{i=1}^{n}\alpha_{i}\chi_{A_{i}}$ and  $|\{x\in \mathbb{R} :g(x)\not=0 \}|< \infty$, where the $\alpha_{i}$ are real numbers and the $A_{i}$ are measurable subsets of $\mathbb{R}$) such that
 $\|f-g\|_{L^{\Phi}}^{lux} \leq \frac{\varepsilon}{2}.$
There exists $h\in C_{c}(\mathbb{R})$ such that  
$$ \|h\|_{\infty} \leq \|g\|_{\infty} \hspace*{0.5cm} \text{and} \hspace*{0.5cm} |\{x\in \mathbb{R} :g(x)\not=h(x) \}|< \frac{1}{\Phi\left(\frac{4\|g\|_{\infty}}{\varepsilon}\right)}, $$
according to Lusin's theorem (see. \cite[ page 52]{wrudin}). We have
\begin{align*}
\int_{\mathbb{R}}\Phi\left(\frac{|h(x)-g(x)|}{\varepsilon/2}\right)dx &= \int_{\{x\in \mathbb{R} :g(x)\not=h(x) \}}\Phi\left(\frac{|h(x)-g(x)|}{\varepsilon/2}\right)dx\\
&\leq \Phi\left(\frac{2\|g-h\|_{\infty}}{\varepsilon}\right)
|\{x\in \mathbb{R} :g(x)\not=h(x) \}| \\
&\leq \Phi\left(\frac{4\|g\|_{\infty}}{\varepsilon}\right)|\{x\in \mathbb{R} :g(x)\not=h(x) \}| 
\leq 1.
\end{align*}
We deduce that  
$ \|h-g\|_{L^{\Phi}}^{lux} \leq \frac{\varepsilon}{2}.$
It follows that $$\|f-h\|_{L^{\Phi}}^{lux} \leq \|f-g\|_{L^{\Phi}}^{lux} +\|g-h\|_{L^{\Phi}}^{lux} \leq \frac{\varepsilon}{2}+\frac{\varepsilon}{2}= \varepsilon. $$ 
\end{proof}

Let  $\{\varphi_{t}\}_{t> 0}$ be a family of integrable functions on $\mathbb{R}$. We say that $\{\varphi_{t}\}_{t> 0}$  is an approximate identity on $\mathbb{R}$ if the following properties are satisfied
\begin{itemize}
\item[(a)]  for all $t> 0$, for all $u \in \mathbb{R}$, $\varphi_{t}(u) \geq 0$ and  $\int_{\mathbb{R}}\varphi_{t}(u)du=1$, 
\item[(b)] for each fixed $ \delta >0$, 
\begin{equation}\label{eq:sauaqi8n}
\lim_{t \to 0}\int_{|u |>\delta }\varphi_{t}(u)du =0
\end{equation}
and
\begin{equation}\label{eq:saui8n}
\lim_{t \to 0}\sup_{|u |>\delta}\varphi_{t}(u)=0. \end{equation}
\end{itemize} 
The family $\{P_{y}\}_{y>0}$ is a classic example of an approximate identity on $\mathbb{R}$  \cite[Page 19]{Jbgarnett}.

\begin{thm}\label{pro:main 3vqa2pm}
Let $\Phi$ be a convex growth function. Let $\{\varphi_{t}\}_{t> 0}$ be an approximate identity on $\mathbb{R}$ and  let  $f\in L^{\Phi}(\mathbb{R})$. The following assertions are satisfied.
\begin{itemize}
\item[(i)]  For all $t> 0$,  $(\varphi_{t}\star f) \in L^{\Phi}(\mathbb{R})$ and
$$   \|(\varphi_{t}\star f)\|_{L^{\Phi}}^{lux}   \leq   \|f\|_{L^{\Phi}}^{lux}.     $$
\item[(ii)] Moreover, if  $\Phi\in \Delta_{2}$ then we have
$$  \lim_{t \to 0}\| (\varphi_{t}\star f)-f\|_{L^{\Phi}}^{lux}=0.      $$
\end{itemize}
\end{thm}

\begin{proof}
$i)$ This assertion
follows from Jensen's inequality and Fubbini's theorem.\\
$ii)$ Assume $f\not\equiv 0$ because there is nothing to show when $f\equiv 0$.\\
 Let $0<\varepsilon<1$.  There exists  $g\in  C_{c}(\mathbb{R})$ such that $ \|f-g\|_{L^{\Phi}}^{lux} \leq \frac{\varepsilon}{3}$, according to Proposition \ref{pro:main706}. For $u \in \mathbb{R}$, consider $f_{u}$ and $g_{u}$, the functions defined by
$$   f_{u}(x)=f(x-u) \hspace*{0.5cm}\text{and} \hspace*{0.5cm}g_{u}(x)=g(x-u),~~ \forall~ x \in \mathbb{R}.$$ 
It is immediate that
$$\| f_{u}-g_{u}\|_{L^{\Phi}}^{lux} = \| f-g\|_{L^{\Phi}}^{lux} \hspace*{0.5cm}\text{and} \hspace*{0.5cm} |\{x\in \mathbb{R}: g(x-u)-g(x)\not=0\}| \leq 2 |K|,$$
where $K$ is the support of  $g$. Since  $g$ is uniformly continuous on $\mathbb{R}$, there exists $ \alpha_{\varepsilon} >0$ such that for all $u \in \mathbb{R}$ with $|u| \leq \alpha_{\varepsilon}$,
$$   |g(x-u)-g(x)| \leq \frac{\varepsilon}{3} \Phi^{-1}\left( \frac{1}{2|K|}\right), ~~ \forall~ x \in \mathbb{R}.   $$
It follows that 
\begin{align*}
\int_{\mathbb{R}}\Phi\left(\frac{|g_{u}(x)- g(x)|}{\varepsilon/3}\right)dx&= \int_{\{x\in \mathbb{R}: g(x-u)-g(x)\not=0\}}\Phi\left(\frac{|g(x-u)- g(x)|}{\varepsilon/3}\right)dx\\
&\leq \Phi\left( \Phi^{-1}\left( \frac{1}{2|K|}\right)  \right)|\{x\in \mathbb{R}: g(x-u)-g(x)\not=0\}| \leq 1.
\end{align*}
We deduce that 
$ \|g_{u}-g\|_{L^{\Phi}}^{lux} \leq \frac{\varepsilon}{3}.$ Therefore $$\|f_{u}-f\|_{L^{\Phi}}^{lux} \leq \|f_{u}-g_{u}\|_{L^{\Phi}}^{lux}+\|g_{u}-g\|_{L^{\Phi}}^{lux}+\|g-f\|_{L^{\Phi}}^{lux}\leq \frac{\varepsilon}{3}+\frac{\varepsilon}{3}+\frac{\varepsilon}{3}=  \varepsilon.$$
For $t>0$,  put 
$$ I:= \int_{|u|\leq \alpha_{\varepsilon}}\int_{\mathbb{R}}\varphi_{t}(u)\Phi\left(\frac{|f_{u}(x)- f(x)|}{2\varepsilon}\right)du dx              $$ 
and
$$ II:= \int_{|u|> \alpha_{\varepsilon}}\int_{\mathbb{R}}\varphi_{t}(u)\Phi\left(\frac{|f_{u}(x)- f(x)|}{2\varepsilon}\right)du dx.              $$
By Fubbini's theorem we have
\begin{align*}
I &:= \int_{|u|\leq \alpha_{\varepsilon}}\int_{\mathbb{R}}\varphi_{t}(u)\Phi\left(\frac{|f_{u}(x)- f(x)|}{2\varepsilon}\right)du dx \\
&\leq \frac{1}{2}\int_{|u|\leq \alpha_{\varepsilon}}\varphi_{t}(u)\left(\int_{\mathbb{R}}\Phi\left(\frac{|f_{u}(x)- f(x)|}{\|f_{u}-f\|_{L^{\Phi}}^{lux}}\right)dx \right) du  
\leq \frac{1}{2}\int_{|u|\leq \alpha_{\varepsilon}}\varphi_{t}(u)du \leq \frac{1}{2}.
\end{align*}
Let us now estimate the quantity $II$. From the Relation (\ref{eq:sauaqi8n}), there exists $t_{\varepsilon}>0$ such that 
$$  \int_{|u|> \alpha_{\varepsilon}}\varphi_{t}(u)du \leq \frac{1}{2\int_{\mathbb{R}}\Phi\left(\frac{| f(x)|}{\varepsilon}\right)dx}, ~~ \forall~0<t<t_{\varepsilon}.          $$
For  $0<t<t_{\varepsilon}$, we have
\begin{align*}
II&:=\int_{|u|> \alpha_{\varepsilon}}\int_{\mathbb{R}}\varphi_{t}(u)\Phi\left(\frac{|f_{u}(x)- f(x)|}{2\varepsilon}\right)du dx \\
 &\leq\int_{|u|> \alpha_{\varepsilon}}\varphi_{t}(u)\left(\int_{\mathbb{R}}\Phi\left(\frac{1}{2}\frac{|f_{u}(x)|}{\varepsilon}+\frac{1}{2}\frac{|f(x)|}{\varepsilon}\right)dx \right) du \\ 
&\leq\int_{|u|> \alpha_{\varepsilon}}\varphi_{t}(u)\left( \frac{1}{2}\int_{\mathbb{R}}\Phi\left(\frac{|f_{u}(x)|}{\varepsilon}\right)dx  + \frac{1}{2}\int_{\mathbb{R}}\Phi\left(\frac{| f(x)|}{\varepsilon}\right)dx \right)du \\
&\leq\int_{|u|> \alpha_{\varepsilon}}\varphi_{t}(u)\left( \int_{\mathbb{R}}\Phi\left(\frac{| f(x)|}{\varepsilon}\right)dx\right)du \\
&\leq  \int_{\mathbb{R}}\Phi\left(\frac{| f(x)|}{\varepsilon}\right)dx\times \frac{1}{2\int_{\mathbb{R}}\Phi\left(\frac{| f(x)|}{\varepsilon}\right)dx}=\frac{1}{2}.
\end{align*}
From Jensen's inequality and estimates $I$ and $II$, we deduce that
\begin{align*}
\int_{\mathbb{R}}\Phi\left(\frac{|(\varphi_{t}\star f)(x)- f(x)|}{2\varepsilon}\right)dx  
&\leq \int_{\mathbb{R}}\int_{\mathbb{R}}\varphi_{t}(u)\Phi\left(\frac{|f_{u}(x)-f(x)|}{2\varepsilon}\right)dudx \\
&=I + II 
\leq \frac{1}{2} + \frac{1}{2}=1.
\end{align*}
Therefore,
$$ \| (\varphi_{t}\star f)-f\|_{L^{\Phi}}^{lux} \leq 2\varepsilon,~~ \forall~ 0<t<t_{\varepsilon}.        $$
\end{proof}

If  $\Phi(t)=t^{p}$ with $1 \leq  p< \infty$ then Theorem \ref{pro:main 3vqa2pm} corresponds to   \cite[Theorem 10.11]{javadmas}.
  
\subsection{Some properties of Hardy-Orlicz space on $\mathbb{C_{+}}$.}

Let $f$ be a measurable function on $\mathbb{R}$. The Poisson integral  $U_{f}$ of $f$ and its conjugate  $V_{f}$  are the functions defined respectively by
$$  U_{f}(x+iy):=  \frac{1}{\pi}\int_{\mathbb{R}}\frac{y}{(x-t)^{2}+y^{2}}f(t)dt,~~ \forall~ x+iy \in \mathbb{C_{+}}    $$
and
$$   V_{f}(x+iy):= \frac{1}{\pi}\int_{\mathbb{R}}\frac{x-t}{(x-t)^{2}+y^{2}}f(t)dt,~~ \forall~x+iy \in \mathbb{C}_{+},   $$
when they exist.
If $f\in L^{1}\left(\frac{dt}{1+t^{2}}\right)$ then $U_{f}$ is a harmonic function on $\mathbb{C_{+ }}$ and
  $$  \lim_{y \to 0}U_{f}(x+iy)=f(x),       $$
  for almost all $x\in \mathbb{R}$. Moreover, if  $f\in L^{1}\left( \frac{dt}{1+|t|} \right) \cap L^{1}\left(\frac{dt}{1+t^{2}}\right)$ then $V_{f}$ is the harmonic conjugate of $U_{f}$ and 
  $$ \lim_{y\to 0} V_{f}(x+iy)= \mathcal{H}(f)(x),      $$
for almost all $x\in \mathbb{R}$ (see. \cite{javadmas}).

\begin{cor}\label{pro:main 3v2qqakam}
Let $\Phi$ be a convex growth function. If   $\Phi\in \Delta_{2}$ then for all 
    $f\in L^{\Phi}(\mathbb{R})$, we have  
$$  \lim_{y \to 0}\| U_{f}(.+iy)-f\|_{L^{\Phi}}^{lux}=0.  $$
\end{cor}

\begin{proof}
The proof is an immediate consequence of Theorem \ref{pro:main 3vqa2pm}.
\end{proof}

\begin{thm}\label{pro:main11pqa6}
Let  $\Phi$ be an N$-$function  and  $f \in L^{\Phi}(\mathbb{R})$. If $\Phi \in \nabla_{2}$ then  $V_{f}\in   h^{\Phi}(\mathbb{C}_{+})$,  $\mathcal{H}(f) \in L^{\Phi}(\mathbb{R})$ and there is a constant $C:=C_{\Phi}>0$ such that
 \begin{equation}\label{eq:5aqqx656}
 \|V_{f}\|_{h^{\Phi}}^{lux}=\| \mathcal{H}(f)\|_{L^{\Phi}}^{lux}\leq C\| f\|_{L^{\Phi}}^{lux}.   \end{equation}
Moreover,   
 $$  V_{f}(x+iy)=U_{\mathcal{H}(f)}(x+iy), ~~ \forall~x+iy \in \mathbb{C_{+}}         $$ 
and
$$  \lim_{y \to 0}\|V_{f}(.+iy)-\mathcal{H}(f)\|_{L^{\Phi}}^{lux}=0.        $$
\end{thm}

For the proof of the Theorem \ref{pro:main11pqa6}, we need the following lemma.

\begin{lem}\label{pro:main 3aqqq2}
Let $\Phi$ be a convex growth function such that $\Phi(t)>0$ for all  $t>0$ and $f\in L^{1}\left(\frac{dt}{1+t^{2}}\right)$. The following assertions are equivalent.
\begin{itemize}
 \item[(i)] $f\in L^{\Phi}(\mathbb{R})$.
 \item[(ii)]  $U_{f}\in h^{\Phi}(\mathbb{C}_{+})$.
 \end{itemize}
 Moreover, $$   \|U_{f}\|_{h^{\Phi}}^{lux}  = \|f\|_{L^{\Phi}}^{lux}.  $$
 \end{lem}

\begin{proof} 
We assume $f\not=0$ because there is nothing to show otherwise.\\
Show that $i)$ implies $ii)$.
By definition, we have
  $$U_{f}(x+iy)=\left( P_{y}\star f\right)(x),~ \forall~ x+iy \in \mathbb{C_{+}}. $$
Since $\{P_{y}\}_{y>0}$ is an approximate identity on $\mathbb{R}$ we have
$$ \sup_{y> 0} \|(P_{y}\star f)\|_{L^{\Phi}}^{lux}   \leq   \|f\|_{L^{\Phi}}^{lux},
      $$
thanks to point (i) of Theorem \ref{pro:main 3vqa2pm}. It follows that
\begin{equation}\label{eq:65f6}
  \|U_{f}\|_{h^{\Phi}}^{lux} =\sup_{y> 0} \|(P_{y}\star f)\|_{L^{\Phi}}^{lux}   \leq   \|f\|_{L^{\Phi}}^{lux}.    \end{equation}
Which allows us to say that  $U_{f}\in h^{\Phi}(\mathbb{C}_{+})$ since $U_{f}$ is harmonic on $\mathbb{C_{+}}$.\\
 Let's prove the converse. Since 
 $$ \lim_{y \to 0}U_{f}(x+iy)=f(x),      $$
for almost all $x\in \mathbb{R}$, we have 
 $$  \int_{\mathbb{R}}\Phi\left(\frac{|f(x)|}{\|U_{f}\|_{h^{\Phi}}^{lux}}\right)dx \leq \liminf_{y\to 0}\int_{\mathbb{R}}\Phi\left(\frac{|U_{f}(x+iy)|}{\|U_{f}\|_{h^{\Phi}}^{lux}}\right)dx \leq \sup_{y> 0} \int_{\mathbb{R}}\Phi\left(\frac{|U_{f}(x+iy)|}{\|U_{f}(.+iy)\|_{L^{\Phi}}^{lux}}\right)dx \leq 1,  $$
thanks to Fatou's lemma.  It follows that  $f\in L^{\Phi}(\mathbb{R})$ and 
 \begin{equation}\label{eq:aqf6}
 \|f\|_{L^{\Phi}}^{lux}\leq \|U_{f}\|_{h^{\Phi}}^{lux}.    
 \end{equation}
We deduce that 
 $$  \|U_{f}\|_{h^{\Phi}}^{lux} = \|f\|_{L^{\Phi}}^{lux},      $$
 thanks to  Relations (\ref{eq:65f6}) and (\ref{eq:aqf6}).
 \end{proof} 
 
It follows from the Lemma \ref{pro:main 3aqqq2} that the map $f\longmapsto U_{f}$ is a linear isometry of $L^{\Phi}(\mathbb{ R})$ into $h^{\Phi}(\mathbb{C}_{+})$.

 \proof[Proof of Theorem \ref{pro:main11pqa6}]
We assume $f\not=0$ because there is nothing to show otherwise. Since $\mathcal{H}$ is  bounded on $L^{\Phi}(\mathbb{ R})$, there exists $C>0$ a constant independent of $f$ such that
$$ \| \mathcal{H}(f)\|_{L^{\Phi}}^{lux} \leq C \| f\|_{L^{\Phi}}^{lux},  $$
(see. \cite[Theorem 3.1.4]{kokokrbec}).   It follows that
$$ \|U_{\mathcal{H}(f)}\|_{h^{\Phi}}^{lux}=\| \mathcal{H}(f)\|_{L^{\Phi}}^{lux}  $$
and 
 $$ \lim_{y\to 0}\| U_{\mathcal{H}(f)}(.+iy)-\mathcal{H}(f)\|_{L^{\Phi}}^{lux}=0,        $$
thanks to Lemma \ref{pro:main 3aqqq2} and Corollary \ref{pro:main 3v2qqakam}.
 To have  $V_{f}\in   h^{\Phi}(\mathbb{C}_{+})$ we need only prove that
$$ V_{f}(x+iy)=U_{\mathcal{H}(f)}(x+iy), ~~ \forall~x+iy \in \mathbb{C_{+}}.    $$
Since $\Phi$ and its complementary function $\Psi$ both satisfy $\nabla_2-$condition, we have $1< a_\Phi \leq b_\Phi<\infty$ and $1< a_\Psi \leq b_\Psi<\infty$. 
As  $L^{a_\Phi }(\mathbb{R}) \cap L^{b_\Phi}(\mathbb{R})$ is a dense subset in $L^{\Phi}(\mathbb{R})$, in view of the Proposition \ref{pro:main706}, there exists a sequence $\{f_{n}\}_{n}$ in $L^{a_\Phi }(\mathbb{R}) \cap L^{b_\Phi}(\mathbb{R})$ such that
$$    \lim_{n \to \infty}\|f_{n}-f\|_{L^{\Phi}}^{lux}=0.   $$
For $n\in \mathbb{N}$, we have $\mathcal{H}(f_{n}) \in L^{a_\Phi}(\mathbb{R}) \cap L^{b_\Phi}(\mathbb{R})$ and
$$ V_{f_{n}}(x+iy)=U_{\mathcal{H}(f_{n})}(x+iy), ~~ \forall~x+iy \in \mathbb{C_{+}},      $$
(see. \cite[Theorem 14.7]{javadmas}).
For $x+iy \in \mathbb{C_{+}}$,  it follows that
\begin{align*}
 \left|V_{f}(x+iy)-U_{\mathcal{H}(f)}(x+iy)\right|
 &= |(Q_{y}\star(f-f_{n}) )(x) + (P_{y}\star(\mathcal{H}(f_{n}-f)) )(x) | \\
 &\lesssim \|Q_{y}\|_{L^{\Psi}}^{lux}\|f_{n}-f\|_{L^{\Phi}}^{lux} + \|P_{y}\|_{L^{\Psi}}^{lux}\|\mathcal{H}(f_{n}-f)) \|_{L^{\Phi}}^{lux} \\
 &\lesssim \|Q_{y}\|_{L^{\Psi}}^{lux}\|f_{n}-f\|_{L^{\Phi}}^{lux} + \|P_{y}\|_{L^{\Psi}}^{lux}\|f_{n}-f \|_{L^{\Phi}}^{lux},
 \end{align*} 
by H\"older's inequality in Orlicz spaces. We deduce that 
 $$ V_{f}(x+iy)=U_{\mathcal{H}(f)}(x+iy), ~~ \forall~x+iy \in \mathbb{C_{+}},     $$
 when $n \longrightarrow \infty$.
\epf

\begin{prop}\label{pro:main6mq4}
Let $\Phi$ be a convex growth function such that  $\Phi(t)>0$ for all  $t>0$. For $F \in h^{\Phi}(\mathbb{C}_{+})$, we have
 \begin{equation}\label{eq:ualp5leson}
 |F(x+iy)|\leq \Phi^{-1}\left(\frac{2}{\pi y}\right)\|F\|_{h^{\Phi}}^{lux}, ~~ \forall~x+iy \in \mathbb{C_{+}}.\end{equation}
 \end{prop}

 \begin{proof}
Let $0\not\equiv F \in h^{\Phi}(\mathbb{C}_{+})$. 
For $z_{0}=x_{0}+iy_{0}\in \mathbb{C_{+}}$ and   $0<r<y_{0}$, using
$$ F(z_{0})= \frac{1}{\pi r^{2}}\int \int_{\overline{\mathcal{D}(z_{0}, r)}}F(u+iv)dudv,        $$
where $\mathcal{D}(z_{0}, r)$ is the disk centered at $z_{0}$ and of radius $r$,  and  Jensen's inequality we have
\begin{align*}
 \Phi\left(\frac{|F(z_{0})|}{\|F\|_{h^{\Phi}}^{lux}}\right) &\leq \Phi\left( \frac{1}{\pi r^{2}}\int \int_{\overline{\mathcal{D}(z_{0}, r)}}\frac{|F(u+iv)|}{\|F\|_{h^{\Phi}}^{lux}}dudv \right)\\
 &\leq\frac{1}{\pi r^{2}}\int \int_{\overline{\mathcal{D}(z_{0}, r)}}\Phi\left(\frac{|F(u+iv)|}{\|F\|_{h^{\Phi}}^{lux}}\right)dudv\\ 
 &\leq \frac{1}{\pi r^{2}}\int_{0}^{2r}\int_{\mathbb{R}}\Phi\left(\dfrac{|F(u+iv)|}{\|F\|_{h^{\Phi}}^{lux}}\right)du dv  
 \leq  \frac{1}{\pi r^{2}}\int_{0}^{2r}dv = \frac{2}{\pi r}.
 \end{align*}
We deduce that $$ \Phi\left(\frac{|F(z_{0})|}{\|F\|_{h^{\Phi}}^{lux}}\right) \leq  \frac{2}{\pi r}, ~~\forall~r<y_{0}.       $$
 \end{proof}

\begin{lem}\label{pro:main6aapaqmq4k}
Let $\Phi$  a convex growth function such that $\Phi(t)>0$ for all  $t>0$ and  $F \in h^{\Phi}(\mathbb{C}_{+})$. For  $\beta>0$, consider $F_{\beta}$ the function defined by   
$$ F_{\beta}(z)=F(z+i\beta),~~\forall~ z \in \mathbb{C}_{+}.    $$
Then $F_{\beta}$ is continuous and bounded on  $\overline{\mathbb{C}_{+}}:=\mathbb{C}_{+}\cup \mathbb{R}$ and harmonic on $\mathbb{C_{+}}$.     Moreover, 
$$  F(z+i\beta)= \frac{1}{\pi}\int_{\mathbb{R}}\frac{y}{(x-t)^{2}+y^{2}}F(t+i\beta)dt,~~\forall~ z=x+iy \in \mathbb{C_{+}}.      $$
\end{lem}

\begin{proof}
By construction $F_{\beta}$ is continuous on $\overline{\mathbb{C_{+}}}$ and harmonic on $\mathbb{C_{+}}$. For  $x+iy \in \overline{\mathbb{C_{+}}}$,  we have
 $$ |F_{\beta}(x+iy)|=|F(x+i(y+\beta))|\leq \Phi^{-1}\left(\frac{2}{\pi (y+\beta)}\right)\|F\|_{h^{\Phi}}^{lux} \leq \Phi^{-1}\left(\frac{2}{\pi \beta}\right)\|F\|_{h^{\Phi}}^{lux}.      $$ 
We deduce that $F_{\beta}$ is bounded on $\overline{\mathbb{C_{+}}}$, according to  Proposition \ref{pro:main6mq4}. It follows that
   $$F_{\beta}(x+iy)=\int_{\mathbb{R}}P_{y}(x-t)F_{\beta}(t)dt = \int_{\mathbb{R}}P_{y}(x-t)F(t+i\beta)dt,~~\forall~ x+iy \in \mathbb{C_{+}}, $$
(see \cite[Lemma 3.4]{Jbgarnett}).
\end{proof}

\begin{cor}\label{pro:main6apaqmq4}
Let $\Phi$ be a convex growth function such that  $\Phi(t)>0$ for all  $t>0$. For  $F \in h^{\Phi}(\mathbb{C}_{+})$, the function $y\mapsto \|F(.+iy)\|_{L^{\Phi}}^{lux}$ is decreasing on $\mathbb{R_{+}^{*}}$. Moreover,
$$ \|F\|_{h^{\Phi}}^{lux}= \sup_{y>0}\|F(.+iy)\|_{L^{\Phi}}^{lux}=    \lim_{y \to 0}\|F(.+iy)\|_{L^{\Phi}}^{lux}.    $$
\end{cor}

\begin{proof}
The proof  follows from Lemma \ref{pro:main6aapaqmq4k} and Jensen's inequality.
\end{proof}

We recall that the complex unit disk and the unit circle are respectively defined by
$$  \mathbb{D}:=\{ z\in \mathbb{C} : |z |<1     \} \hspace*{1cm}\text{and} \hspace*{1cm} \mathbb{T}:=\{ z\in \mathbb{C} : |z |=1     \}. 
     $$

Let $0<p<\infty$. The Hardy space on $\mathbb{D}$, $h^{p}(\mathbb{D})$ is the set of harmonic function $F$ on  $\mathbb{D}$ which satisfy 
$$  \|F\|_{h^{p}(\mathbb{D})}:=\sup_{0\leq r<1} \left( \frac{1}{2\pi}\int_{0}^{2\pi}|F(re^{i\theta})|^{p}d\theta \right)^{1/p}<\infty.$$
For $1\leq p<\infty$, $\left( h^{p}(\mathbb{D}),  \|.\|_{h^{p}(\mathbb{D})} \right)$ is a Banach space.

The Cayley transform is the map $\varphi$ defined by
$$   \varphi(\omega)= i\frac{1-\omega}{1+\omega},~~\forall~\omega \in \mathbb{D} \cup \mathbb{T}\backslash\{-1\}. 
    $$
Note that the restriction of $\varphi$ to $\mathbb{D}$ (resp. $\mathbb{T}\backslash\{-1\}$) is an analytic function on $\mathbb{D}$ with values in $\mathbb {C}_{+}$ (resp. a homeomorphism from $\mathbb{T}\backslash\{-1\}$ onto $\mathbb{R}$).

\begin{thm}\label{pro:main5Qpm5}
Let $\Phi$ be a convex growth function such that  $\Phi(t)>0$ for all  $t>0$.  For $F\in h^{\Phi}(\mathbb{C_{+}})$, the function $G$ defined by
$$   G(\omega)= F\left(i\frac{1-\omega}{1+\omega} \right), ~~\forall~\omega \in \mathbb{D},    $$
is in $h^{1}(\mathbb{D})$. Moreover,
$$  \|G\|_{h^{1}(\mathbb{D})} \leq \Phi^{-1}(1) \|F\|_{h^{\Phi}(\mathbb{C_{+}})}^{lux}.     $$
\end{thm}

 \begin{proof}
Let $ 0\not\equiv F\in h^{\Phi}(\mathbb{C_{+}})$. Consider  a decreasing sequence $\{\beta_{n}\}_{n \in \mathbb{N}}$ of
positive reals such that $\lim_{n\to \infty}\beta_{n}=0$. \\
For $n \in \mathbb{N}$, consider $F_{n}$ the function defined by
$$ F_{n}(z)=F(z+i\beta_{n}),~~\forall~ z \in \mathbb{C}_{+}.     $$
Using Lemma \ref{pro:main6aapaqmq4k}  and  Jensen's inequality, we deduce that
\begin{equation}\label{eq:suiaq8n}
\Phi\left(\frac{|F(z+i\beta_{n})|}{\|F\|_{h^{\Phi}}^{lux}}\right)\leq\int_{\mathbb{R}}P_{y}(x-t)\Phi\left(\frac{|F(t+i\beta_{n})|}{\|F\|_{h^{\Phi}}^{lux}}\right)dt, ~~\forall~ z=x+iy \in \mathbb{C_{+}}.\end{equation}
Consider $f_{n}$ the function defined by 
$$    f_{n}(t)= \Phi\left(\frac{|F(t+i\beta_{n})|}{\|F\|_{h^{\Phi}}^{lux}}\right),  ~~ \forall~t \in \mathbb{R}.       $$
By construction the sequence $\{f_{n}\}_{n \in \mathbb{N}}$ is bounded in $ L^{1}(\mathbb{R})$. Indeed, 
$$ \int_{\mathbb{R}}|f_{n}(t)|dt =\int_{\mathbb{R}}\Phi\left(\frac{|F(t+i\beta_{n})|}{\|F\|_{h^{\Phi}}^{lux}}\right)dt \leq \sup_{y>0}\int_{\mathbb{R}}\Phi\left(\frac{|F(t+iy)|}{\|F\|_{h^{\Phi}}^{lux}}\right)dt \leq 1.        $$
Thus there exists $\{f_{n_{k}}\}_{k\in \mathbb{N}}$ a sub$-$sequence of $\{f_{n}\}_{n\in \mathbb{N}}$ and $\lambda$ a positive Borel measure on $\mathbb{R}$ such that 
 $\{f_{n_{k}}\}_{k\in \mathbb{N}}$  converges to $\lambda$ for the topology *$-$weak and 
 $$\lambda(\mathbb{R}) \leq \liminf_{k \to \infty}\|f_{n_{k}}\|_{L^{1}} \leq 1. $$
Since $P_{y} \in \mathcal{C}_{0}(\mathbb{R})$, by letting $k$ tend to infinity in the Relation (\ref{eq:suiaq8n}), we deduce  that 
$$  \Phi\left(\frac{|F(x+iy)|}{\|F\|_{h^{\Phi}}^{lux}}\right)\leq \int_{\mathbb{R}}P_{y}(x-t)d\lambda(t), ~~\forall~ x+iy \in \mathbb{C_{+}},
      $$
where $\mathcal{C}_{0}(\mathbb{R})$ is the set of function $f$ continuous on $\mathbb{R}$ such that  $\lim_{|t|\to \infty} f(t)=0$.\\
Put for $A \in \mathcal{B}(\mathbb{T}\backslash\{-1\})$,
$$    \nu(A)= \lambda\left(  \varphi(A) \right) \hspace*{0.5cm}\text{and} \hspace*{0.5cm}  \nu\left( \{-1\}\right)=0,    $$
where $\mathcal{B}(\mathbb{T}\backslash\{-1\})$ is the tribu of borelians on $\mathbb{T}\backslash\{-1\}$. Since $\lambda$ is a positive measure on $\mathbb{R}$ and $\varphi$ a homeomorphism from $\mathbb{T}\backslash\{-1\}$ onto $\mathbb{R} $, we deduce $\nu$ is a positive measure on $\mathbb{T}$. Moreover, 
$$ \nu(\mathbb{T})=\int_{\mathbb{T}\backslash\{-1\}}d\nu(e^{i\theta}) = \int_{\mathbb{R}} d\lambda(t)= \lambda(\mathbb{R}) \leq 1.   $$
Fix $0\leq r<1$. For  $\omega=re^{iu} \in \mathbb{D}$ and $z=x+iy\in \mathbb{C}_{+}$ such that $ z=i\frac{1-\omega}{1+\omega}$, we have 
\begin{align*}
\Phi\left(\frac{|G(re^{iu})|}{\|F\|_{h^{\Phi}}^{lux}}\right) &=  \Phi\left(\frac{|F\circ \varphi(re^{iu})|}{\|F\|_{h^{\Phi}}^{lux}}\right)  = \Phi\left(\frac{|F(x+iy)|}{\|F\|_{h^{\Phi}}^{lux}}\right)  \\
&\leq  \frac{1}{\pi}\int_{\mathbb{R}}\frac{y}{(x-t)^{2}+y^{2}}d\lambda(t) \leq  \frac{1}{\pi}\int_{\mathbb{R}}\frac{y}{(x-t)^{2}+y^{2}}(1+t^{2})d\lambda(t)     \\
&= \frac{1}{\pi}\int_{\mathbb{T}}\frac{1-r^{2}}{1-2r\cos(u-\theta)+r^{2}}d\nu(e^{i\theta}).
\end{align*}
It follows that
\begin{align*}
\Phi\left(\frac{1}{2\pi}\int_{-\pi}^{\pi}\frac{|G(re^{iu})|}{\|F\|_{h^{\Phi}}^{lux}}du \right)&\leq\frac{1}{2\pi}\int_{-\pi}^{\pi}\Phi\left(\frac{|G(re^{iu})|}{\|F\|_{h^{\Phi}}^{lux}}\right)du \\
 &\leq \frac{1}{2\pi}\int_{-\pi}^{\pi}\int_{\mathbb{T}}\frac{1-r^{2}}{1-2r\cos(u-\theta)+r^{2}}d\nu(e^{i\theta})du\\
&=\int_{\mathbb{T}}\left( \frac{1}{2\pi}\int_{-\pi}^{\pi}\frac{1-r^{2}}{1-2r\cos(u-\theta)+r^{2}}du  \right) d\nu(e^{i\theta})\\
&=\int_{\mathbb{T}}d\nu(e^{i\theta})= \nu(\mathbb{T}) \leq 1.
\end{align*}
\end{proof}

\begin{cor}\label{pro:main5Qpmpm5}
Let $\Phi$ be a convex growth function such that  $\Phi(t)>0$ for all  $t>0$.  For all $F\in H^{\Phi}(\mathbb{C_{+}})$, the function $G$ defined by
$$   G(\omega)= F\left(i\frac{1-\omega}{1+\omega} \right), ~~\forall~\omega \in \mathbb{D},    $$
is in $H^{1}(\mathbb{D})$. Moreover,
$$   \|G\|_{H^{1}(\mathbb{D})} \leq \Phi^{-1}(1) \|F\|_{H^{\Phi}(\mathbb{C_{+}})}^{lux},
    $$
where  $H^{1}(\mathbb{D})$ is the subspace of  $h^{1}(\mathbb{D})$ consisting of analytic functions on  $\mathbb{D}$.    
\end{cor}

\subsection{Control of some maximal operators.} 

Let  $E$ be a measurable set of $\mathbb{R}$. We denote by
$$ |E|:=\int_{E}dx.    $$

For $\beta \in \left\{0; 1/3\right\}$, we call interval $\beta-$dyadic any interval $I$ of $\mathbb{R}$ of the form
 $$   2^{-j}(\lbrack 0,1) +k+(-1)^{j}\beta),    $$
 where $k,j\in \mathbb{Z}$. We put $\mathcal{D}_{j}^{\beta}:=\cup_{k \in \mathbb{Z}}[2^{-j}(\lbrack 0,1) +k+(-1)^{j}\beta)]$, for all $j\in \mathbb{Z}$
  and $\mathcal{D}^{\beta}:=\bigcup_{j\in \mathbb{Z}}\mathcal{D}_{j}^{\beta}$.
 We have the following properties (see for example \cite{DVCruzJMMartell},\cite{Spmcregu},\cite{Emstshae}):
 \begin{itemize}
 \item[-] for all $I,J\in \mathcal{D}^{\beta}$, we have $I\cap J \in \left\{\emptyset ; I; J\right\}$,
 \item[-] for each fixed  $j\in \mathbb{Z}$, if $I \in \mathcal{D}_{j}^{\beta}$ then there exists a unique $J \in \mathcal{D}_{j-1}^{\beta}$ such that  $I\subset J$,
 \item[-] for each fixed $j\in \mathbb{Z}$, if $I \in \mathcal{D}_{j}^{\beta}$ then there exists $I_{1}, I_{2} \in \mathcal{D}_{j+1}^{\beta}$ such that
  $I=I_{1}\cup I_{2}$ and $I_{1}\cap I_{2}=\emptyset$.
 \end{itemize}

We can find the following result in \cite{jmtafeuseh}.
 
 \begin{lem}\label{pro:main80}
 Let $I$ be an interval.  There exists    $\beta \in \left\{0, 1/3\right\}$ and $J \in \mathcal{D}^{\beta}$ such that $I\subset J $ and $|J|\leq 6|I|$.
 \end{lem}

Let $f$ be a measurable function on $\mathbb{R}$. We define the maximal Hardy-Littlewood function,  $\mathcal{M}_{HL}(f)$ of $f$ by
 $$  \mathcal{M}_{HL}(f)(x):=\sup_{I \subset \mathbb{R} }\frac{\chi_{I}(x)}{|I|}\int_{I}| f(t)| dt, ~~\forall~ x \in \mathbb{R},      $$
 where the supremum is taken over all intervals of $\mathbb{R}$. Similarly, for  $\beta \in \left\{0; 1/3\right\}$, we define its dyadic version $\mathcal{M}^{\mathcal{D}^{\beta}}_{HL}(f)$,  where the supremum is taken this time on the intervals contained in $\mathcal{D}^{\beta}$. We have
 \begin{equation}\label{eq:of895ua1l}
  \mathcal{M}_{HL}(f) \leq 6 \sum_{\beta \in \{0; 1/3\} }   \mathcal{M}^{\mathcal{D}^{\beta}}_{HL}(f).
 \end{equation}

For  $\beta \in \left\{0, 1/3\right\}$,  $\lambda > 0$ and  $f$ be a measurable function on $\mathbb{R}$,  consider $\mathcal{D}_{\lambda}(f)$ (resp. $\mathcal{D}_{\lambda}^{*}(f)$) the set of intervals $I$,  $\beta-$dyadic (resp. maximal for the inclusion) satisfying the relation
$$  \frac{1}{|I|}\int_{I}|f(y)| dy > \lambda.  $$
If $\mathcal{D}_{\lambda}^{*}(f)\not=\emptyset$ then its elements are pairwise disjoint. Moreover, each
element of $\mathcal{D}_{\lambda}(f)$ is contained in an element of $\mathcal{D}_{\lambda}^{*}(f)$.
If $f\in L^{1}_{loc}(\mathbb{R})$ and $\frac{1}{|I|}\int_{I}|f(y)| dy \longrightarrow 0$ when $|I| \longrightarrow \infty$ then $\mathcal{D}_{\lambda}^{*}(f)\not=\emptyset$ whenever $\mathcal{D}_{\lambda}(f)\not=\emptyset$. Moreover,
$$  \left\{x\in \mathbb{R} : \mathcal{M}^{\mathcal{D}^{\beta}}_{HL}(f)(x)> \lambda \right\} = \bigcup_{I \in \mathcal{D}_{\lambda}^{*}(f)}I 
       $$
and 
$$  \lambda <  \frac{1}{|I|}\int_{I}|f(y)| dy     \leq 2 \lambda, ~~ \forall~I \in \mathcal{D}_{\lambda}^{*}(f).    $$
We have the following inequality
\begin{equation}\label{eq:foaaqimal}
\frac{1}{\lambda}\int_{\{x\in \mathbb{R}~ :~ |f(x)|> \lambda \}}|f(x)|dx  \leq \left|\left\{x\in \mathbb{R} : \mathcal{M}_{HL}(f)(x)> \frac{\lambda}{12} \right\}\right|. 
\end{equation}

Let $\alpha \geq 0$ and  $f$  a measurable function on $\mathbb{R}$. If  $f\in L^{1}\left(\frac{dt}{1+t^{2}}\right)$ then we have, 
\begin{equation}\label{eq:claqis}
 \frac{1}{2\pi}\mathcal{M}_{HL}(f)(x) \leq \mathcal{M}_{rad}(U_{f})(x) \leq \mathcal{M}_{ntg}^{\alpha}(U_{f})(x) \leq \left( 1+ \frac{2 \alpha}{\pi} \right) \mathcal{M}_{HL}(f)(x), ~~\forall~ x \in \mathbb{R}.
\end{equation}
Moreover, if $f\in L^{1}\left(  \frac{dt}{1+|t|}  \right) \cap L^{1}\left(\frac{dt}{1+t^{2}}\right)$ then we have also,
\begin{equation}\label{eq:soaqn}
\mathcal{\widetilde{H}}(f)(x)  \leq \left( 1+ 1/\pi \right)\mathcal{M}_{HL}(f)(x) + \mathcal{M}_{rad}(V_{f})(x), ~~ \forall~x \in \mathbb{R},\end{equation}
(see. \cite{javadmas}).

\begin{thm}\label{pro:main 5dmpmw3pl}
  Let  $\alpha \geq 0$ and $\Phi$  an  N-function such that $\Phi \in \Delta_{2}$. The following assertions are equivalent.
\begin{itemize}
 \item[(i)] $\Phi\in \nabla_{2}$.
 \item[(ii)] There exists a constant $C_{1}>0$ such that for all $f\in L^{\Phi}(\mathbb{R})$,   
   \begin{equation}\label{eq:cjlaqaaqqis}
   \|\mathcal{M}_{HL}(f)\|_{L^{\Phi}}^{lux}  \leq C_{1}\|f\|_{L^{\Phi}}^{lux}.
   \end{equation}
 \item[(iii)] There exists a constant $C_{2}>0$ such that for all $f\in L^{\Phi}(\mathbb{R})$,   
  \begin{equation}\label{eq:cjlaaqqis}
\|f\|_{L^{\Phi}}^{lux} \leq \|\mathcal{M}_{rad}(U_{f})\|_{L^{\Phi}}^{lux} \leq\|\mathcal{M}_{ntg}^{\alpha}(U_{f})\|_{L^{\Phi}}^{lux}  \leq C_{2}\|f\|_{L^{\Phi}}^{lux}.
  \end{equation}
 \item[(iv)] There exists a constant $C_{3}>0$ such that for all $f\in L^{\Phi}(\mathbb{R})$,   
  \begin{equation}\label{eq:cotjplaqlpmis}
  \|\mathcal{\widetilde{H}}(f)\|_{L^{\Phi}}^{lux} \leq C_{3}\|f\|_{L^{\Phi}}^{lux}.
  \end{equation}
 \end{itemize}
 \end{thm}

For the proof of Theorem \ref{pro:main 5dmpmw3pl}, we need the following results.

\begin{lem}\label{pro:mainplaaqq6}
Let  $\Phi_{1}, \Phi_{2} \in \mathscr{U}$. The following assertions are equivalent.
 \begin{itemize}
\item[(i)] There exists a constant $C_{1}>0$ such that for all $t>0$, 
\begin{equation}\label{eq:conditdedinis}
\int_0^t\frac{\Phi_{2}(s)}{s^{2}}ds\le C_{1}\frac{\Phi_{1}(t)}{t}.\end{equation}
\item[(ii)] There exists a constant $C_{2}>0$ such that for all $f\in L^{\Phi_{1}}(\mathbb{R})$, 
\begin{equation}\label{eq:fonciniaqmal}
\|\mathcal{M}_{HL}(f)\|_{L^{\Phi_{2}}}^{lux}  \leq C_{2}\|f\|_{L^{\Phi_{1}}}^{lux}.
\end{equation}
\item[(iii)] There exists a constant $C_{3}>0$ such that for all $f\in L^{\Phi_{1}}(\mathbb{R})$   
  \begin{equation}\label{eq:cjlaaqqis}
\|f\|_{L^{\Phi_{2}}}^{lux} \leq \|\mathcal{M}_{rad}(U_{f})\|_{L^{\Phi_{2}}}^{lux} \leq\|\mathcal{M}_{ntg}^{\alpha}(U_{f})\|_{L^{\Phi_{2}}}^{lux}  \leq C_{3}\|f\|_{L^{\Phi_{1}}}^{lux}.
  \end{equation}
 \end{itemize}
\end{lem}

\begin{prop}\label{pro:mainplaqaqaq6}
Let $\Phi$ be an  N-function such that $\Phi \in \Delta_{2}$. Then the following assertions are equivalent.
\begin{itemize}
\item[(i)] $\Phi \in \nabla_{2}$.
\item[(ii)] $\mathcal{M}_{HL}: L^{\Phi}(\mathbb{R})\longrightarrow L^{\Phi}(\mathbb{R})$ is bounded.
\item[(iii)] $\mathcal{M}^{\mathcal{D}^{\beta}}_{HL}: L^{\Phi}(\mathbb{R})\longrightarrow L^{\Phi}(\mathbb{R})$ is bounded, for all  $\beta \in \{0; 1/3\}$.
\item[(iv)] There exists a constant $C>0$ such that for all $f\in L^{\Phi}(\mathbb{R})$   
  \begin{equation}\label{eq:cjlaaqqis}
\|f\|_{L^{\Phi}}^{lux} \leq \|\mathcal{M}_{rad}(U_{f})\|_{L^{\Phi}}^{lux} \leq\|\mathcal{M}_{ntg}^{\alpha}(U_{f})\|_{L^{\Phi}}^{lux}  \leq C\|f\|_{L^{\Phi}}^{lux}.
  \end{equation}
\end{itemize}
\end{prop} 

\begin{proof}
The proof  follows from Relation (\ref{eq:saaui8n}) and Lemma \ref{pro:mainplaaqq6}.
\end{proof}

\proof[Proof of Theorem \ref{pro:main 5dmpmw3pl}]
The equivalence $(i) \Leftrightarrow   (ii)$ follows from Proposition \ref{pro:mainplaqaqaq6}. The equivalence between the points $(ii)$ and $(iii)$ follows from the Relation (\ref{eq:claqis}) and
from Lemma \ref{pro:mainplaaqq6}, (for $\Phi_{1}= \Phi_{2}$). For $(i)$ implies $(ii)$, the proof follows from Relation (\ref{eq:soaqn}) and from Theorem \ref{pro:main11pqa6}. It remains for us to prove $(iv)$ implies $i)$.
Suppose that $\Phi\not\in \nabla_{2}$. There exist  $\beta \in \{0; 1/3\}$ and  $f\in L^{\Phi}(\mathbb{R })$ such that $\mathcal{M}_{HL}^{\mathcal{D}^{\beta}}(f) \not\in L^{\Phi}(\mathbb{R})$, according to Proposition \ref{pro:mainplaqaqaq6}.  
For $\lambda>0$, there exists a family $\{I_{j}\}_{j\in \mathbb{N}}$ of pairwise disjoint $\beta-$dyadic intervals such that 
$$  \left\{x\in \mathbb{R} : \mathcal{M}_{HL}^{\mathcal{D}^{\beta}}(f)(x)> \lambda \right\} = \bigcup_{j\in \mathbb{N}}I_{j}  
      $$
and 
 $$ \lambda <\frac{1}{|I_{j}|}\int_{I_{j}}|f(y)|dy \leq 2 \lambda,  ~~ \forall~j\in \mathbb{N}.      $$
 For  $j\in \mathbb{N}$,  consider the interval $\widetilde{I}_{j}$ defined by
 $$ \widetilde{I}_{j}=I_{j} \cup I_{j}',     $$
where $I_{j}'$ is the interval such that $I_{j} \cap I_{j}'=\emptyset$ and $|I_{j}|=|I_{j}'|$.
For $x\in I_{j}'$, we have 
 $$ \frac{1}{2\pi}\lambda < \frac{1}{2\pi}\frac{1}{|I_{j}|}\int_{I_{j}}|f(t)|dt 
 = \frac{1}{\pi}\frac{1}{|\widetilde{I}_{j}|}\int_{\mathbb{R}}|f(t)\chi_{I_{j}}(t)|dt 
 \leq   \sup_{\varepsilon >0}\left|\frac{1}{\pi}\int_{|x-t|>\varepsilon}\frac{|f(t)\chi_{I_{j}}(t)|}{x-t}dt \right| 
 .
       $$
 We deduce that  
  $$ I_{j}' \subset \left\{x\in \mathbb{R} : \mathcal{\widetilde{H}}\left( |f\chi_{I_{j}}|\right)(x)> \frac{\lambda}{2\pi} \right\}.       $$
Since the inequality (\ref{eq:cotjplaqlis}) is satisfied, 
we have  
 \begin{align*}
 \int_{\mathbb{R}}\Phi\left( \mathcal{M}_{HL}^{\mathcal{D}^{\beta}}(f)(x)\right)dx 
 &= 	\int_{0}^{\infty}\Phi^{\prime}(\lambda)\left|\left\{ x\in \mathbb{R} :   \mathcal{M}_{HL}^{\mathcal{D}^{\beta}}f(x)  > \lambda  \right\}\right|d\lambda \\
 &= 	\int_{0}^{\infty}\Phi^{\prime}(\lambda)\left|\bigcup_{j\in \mathbb{N}}I_{j}\right|d\lambda 
 = \int_{0}^{\infty}\Phi^{\prime}(\lambda)\sum_{j\in \mathbb{N}}\left|I_{j}\right|d\lambda 
 = 
 \sum_{j\in \mathbb{N}}\int_{0}^{\infty}\Phi^{\prime}
 (\lambda)\left|I_{j}'\right|d\lambda\\
 &\lesssim \sum_{j\in \mathbb{N}}\int_{0}^{\infty}\Phi^{\prime}
 (\lambda)\left|\left\{x\in \mathbb{R} : \mathcal{\widetilde{H}}\left( f\chi_{I_{j}}\right)(x)> \frac{\lambda}{2\pi} \right\}\right|d\lambda\\
 &\lesssim \sum_{j\in \mathbb{N}}\int_{\mathbb{R}}\Phi\left( \mathcal{\widetilde{H}}\left( |f\chi_{I_{j}}|\right)(x)\right)dx 
 \lesssim \sum_{j\in \mathbb{N}}\int_{\mathbb{R}}\Phi\left(|f(x)\chi_{I_{j}}(x)|\right)dx\\
 &= \sum_{j\in \mathbb{N}}\int_{I_{j}}\Phi\left(|f(x)|\right)dx= \int_{\bigcup_{j\in \mathbb{N}}I_{j}}\Phi\left(|f(x)|\right)dx < \infty.
  \end{align*}
Which is absurd because  $\mathcal{M}_{HL}^{\mathcal{D}^{\beta}}(f) \not\in L^{\Phi}(\mathbb{R})$. Therefore, $\Phi \in \nabla_{2}$.
\epf

\proof[Proof of Lemma \ref{pro:mainplaaqq6}]
The proof of the equivalence between points $(ii)$ and $(iii)$ follows from the Relation (\ref{eq:claqis}) and Lemma \ref{pro:main 3aqqq2}. For $(i)$ implies $(ii)$ uses the same arguments as the proof of the \cite[Proposition 3.6]{djesehb}. It remains for us to prove the converse.
Suppose $(i)$ is not true.
There exists $(t_{k})_{k \geq 1}$, a sequence of positive real numbers such that
$$  \int_{0}^{2^{k}t_{k}}\frac{\Phi_{2}(s)}{s^{2}}ds\geq \dfrac{2^{k}\Phi_{1}(2^{k}t_{k})}{t_{k}},~~\forall~ k \geq 1.        $$
For $k \geq 1$, consider the function  $f_{k}$ defined by 
$$  f_{k}(x):=2^{k}t_{k}\chi_{I_{k}}(x), ~~\forall~x \in \mathbb{R},   $$
where $\hspace*{0.15cm}  I_{k}:=\left\{x\in \mathbb{R} : a_{k}\leq x<a_{k+1} \right\} \hspace*{0.15cm}\text{with}\hspace*{0.15cm} a_{k}=\sum_{j=0}^{k-1}  \frac{1}{2^{j}\Phi_{1}(2^{j}t_{j})}.$
Since inequality (\ref{eq:foaaqimal}) is satisfied, we have
\begin{align*}
\int_{\mathbb{R}}\Phi_{2}\left( \mathcal{M}_{HL}(12f_{k})(x)\right)dx 
&=\int_{0}^{\infty}\Phi_{2}^{\prime}(\lambda)\left|\left\{x\in \mathbb{R} : \mathcal{M}_{HL}\left(f_{k}\right)(x)> \frac{\lambda}{12} \right\}\right|d\lambda\\
&\gtrsim \int_{0}^{\infty}\Phi_{2}^{\prime}(\lambda) \left( \frac{1}{\lambda}\int_{\{y\in \mathbb{R}~ :~ |f_{k}(y)|> \lambda \}}|f_{k}(x)|dx 
 \right)d\lambda\\
&\gtrsim \int_{I_{k}}2^{k}t_{k}\left(\int_{0}^{2^{k}t_{k}} \frac{\Phi_{2}'(\lambda)}{\lambda}d\lambda \right)dx \\
&\gtrsim  2^{k}t_{k}|I_{k}| \times \dfrac{2^{k}\Phi_{1}(2^{k}t_{k})}{t_{k}}.
\end{align*}
Consider the function  $f$ defined by
$$  f(x)= \sum_{k=1}^{\infty}12 f_{k}(x),~~\forall~ x \in \cup_{k\geq 1}I_{k} \hspace*{0.5cm}\textrm{and} \hspace*{0.5cm}  f(x)=0, ~~\forall~ x \in \mathbb{R} \backslash \cup_{k\geq 1}I_{k}.      $$
Since the  $ \{I_{k}\}_{k} $ are pairwise disjoint and   
$$|I_{k}|= a_{k+1}-a_{k}=\dfrac{1}{2^{k}\Phi_{1}(2^{k}t_{k})}, $$
we deduce that    $f\in L^{\Phi_{1}}(\mathbb{R})$. Indeed,
$$  \int_{\mathbb{R}}\Phi_{1}(|f(x)|)dx\lesssim \sum_{k=1}^{\infty}\int_{I_{k}}\Phi_{1}(2^{k}t_{k})\chi_{I_{k}}(x)dx=\sum_{k=1}^{\infty}\Phi_{1}(2^{k}t_{k})|I_{k}|
= \sum_{k= 1}^{\infty} \dfrac{1}{2^{k}}< \infty.       $$
It follows that
$$ \int_{\mathbb{R}}\Phi_{2}\left( \mathcal{M}_{HL}(f)(x)\right)dx \gtrsim \int_{\mathbb{R}}\Phi_{2}\left( \mathcal{M}_{HL}(12f_{k})(x)\right)dx 
 \gtrsim 2^{k}, ~~\forall~ k \geq 1.
      $$
As $2^{k}$ tends to infinity when $(k\longrightarrow \infty)$, we deduce that  $ \mathcal{M}_{HL}(f) \not\in  L^{\Phi_{2}}(\mathbb{R})$.
\epf


\section{Proofs of main results.}

\subsection{Proof of Theorems \ref{pro:main11a6} and \ref{pro:main11aaq6}.}

\proof[Proof of Theorem \ref{pro:main11a6}.]
The implication $(i) \Rightarrow (ii)$ follows from Lemma \ref{pro:main 3aqqq2}. Let's prove $(i)$ implies $(ii)$.
Suppose that  $F\in h^{\Phi}(\mathbb{C}_{+})$. Since the map $f\mapsto U_{f}$ is injective on $L^{\Phi}(\mathbb{R})$,  uniqueness is trivial in $(ii)$.
We now prove the existence.\\
For $z_{0}=x_{0}+iy_{0} \in \mathbb{C_{+}}$,  consider the sequence $(y_{n})_{n}$  such that $0<y_{n}<y_{0}$ and
decreasing to $0$, and put for $z \in \mathbb{C}_{+}$
$$  F_{n}(z)= F(z+iy_{n}).         $$
By construction, $F_{n}$ is continuous and bounded on $\overline{\mathbb{C}_{+}}$, and harmonic on $\mathbb{ C}_{+}$, according to Lemma \ref{pro:main6aapaqmq4k}. Moreover, we have
\begin{equation}\label{eq:limisjgsion}
F(z_{0})=\lim_{n\to \infty}F(z_{0}+iy_{n})=\lim_{n\to \infty}\int_{\mathbb{R}}P_{y_{0}}(x_{0}-t)F(t+iy_{n})dt.
 \end{equation}
For $n \in \mathbb{N}$,  consider $T_{n}$ the map defined by
$$ T_{n}(g) = \int_{\mathbb{R}}F(t+iy_{n})g(t)dt, ~\forall~ g \in M^{\Psi}(\mathbb{R}), $$
where  $\Psi$ is the complementary function of $\Phi$.  For $g \in M^{\Psi}(\mathbb{R})$, we have
$$ |T_{n}(g)| \leq \int_{\mathbb{R}}|F(t+iy_{n})g(t)|dt\leq 2\|F(.+y_{n})\|_{L^{\Phi}}^{lux}\|g\|_{L^{\Psi}}^{lux}\leq 2\|F\|_{h^{\Phi}}^{lux}\|g\|_{L^{\Psi}}^{lux},       $$
thanks to H\"older's inequality in Orlicz spaces. It follows that the sequence $\{T_{n} \}_{n}$ is bounded in $\left( M^{\Psi}(\mathbb{R})\right)^{* }$. Since $M^{\Psi}(\mathbb{R})$ is a separable Banach space, there exists $T \in \left( M^{\Psi}(\mathbb{R})\right )^{*}$ and a sub$-$sequence $\{T_{n_{k}}\}_{k}$  of the sequence $\{T_{n}\}_{n}$ such that $\{T_{n_{k}}\}_{k}$ converges to $T$ for the topology *$-$weak.  There is a unique function $f\in L^{\Phi}(\mathbb{R})$ such that
 $$  T(g) =\int_{\mathbb{R}}f(x)g(x)dx,~~ \forall~g \in M^{\Psi}(\mathbb{R}),      $$
 since $\left( M^{\Psi}(\mathbb{R})\right)^{*}=L^{\Phi}(\mathbb{R})$. As  $P_{y_{0}}$  belongs to $  M^{\Psi}(\mathbb{R})$, we deduce that
\begin{equation}\label{eq:lion}
\lim_{k\to \infty}\int_{\mathbb{R}}P_{y_{0}}(x_{0}-t)F(t+iy_{n_{k}})dt=\int_{\mathbb{R}}P_{y_{0}}(x_{0}-t)f(t)dt.
\end{equation}
It follows that 
$$  F(z_{0})= \int_{\mathbb{R}}P_{y_{0}}(x_{0}-t)f(t)dt,     $$
thanks to Relations (\ref{eq:limisjgsion}) and (\ref{eq:lion}). Moreover, 
$$  \|F\|_{h^{\Phi}}^{lux}=\lim_{y \to 0}\|F(.+iy)\|_{L^{\Phi}}^{lux}=\| f\|_{L^{\Phi}}^{lux},     $$ 
according to  Lemma \ref{pro:main 3aqqq2} and Corollary \ref{pro:main6apaqmq4}.
\epf

\proof[Proof of Theorem \ref{pro:main11aaq6}.]
Let  $F\in h^{\Phi}(\mathbb{C}_{+})$ and   for $\omega \in \mathbb{D}$, put
$$  G(\omega)= F\left(i\frac{1-\omega}{1+\omega} \right).     $$
Since $G\in h^{1}(\mathbb{D})$, there exists a unique finite Borel measure $\nu$ on $\mathbb{T}$ such that for all $0 \leq r<1$
$$  G(re^{iu})=\int_{\mathbb{T}}\frac{1-r^{2}}{1-2r\cos(\theta-u)+r^{2}}d\nu(e^{i\theta}).      $$
For $A \in \mathcal{B}(\mathbb{R})$ and $t\in \mathbb{R}$ put respectively
$$   \lambda(A)= \nu\left(  \varphi^{-1}(A) \right)   $$
and
$$  d\mu(t) =\pi(1+t^{2})d\lambda(t),    $$
where $\mathcal{B}(\mathbb{R})$ is the tribu of borelians on $\mathbb{R}$. 
By construction, $\lambda$ and $\mu$ are measures on $\mathbb{R}$. Moreover,
$$ \frac{1}{\pi}\int_{\mathbb{R}}\frac{d|\mu|(t)}{1+t^{2}}=|\lambda|(\mathbb{R})= |\nu|\left(  \varphi^{-1}(\mathbb{R}) \right) =  |\nu|(\mathbb{T}\backslash\{-1\})  < \infty.         $$ 
For $z=x+iy\in \mathbb{C}_{+}$ and  $\omega=re^{iu} \in \mathbb{D}$ such that $ z=i\frac{1-\omega}{1+\omega}$, we have 
\begin{align*}
F(x+iy)&= G(re^{iu})=\int_{\mathbb{T}}\frac{1-r^{2}}{1-2r\cos(\theta-u)+r^{2}}d\nu(e^{i\theta})\\
&=  \frac{1-r^{2}}{1+2r\cos u+r^{2}}\nu(\{-1\}) + \int_{\mathbb{T}\backslash\{-1\}}\frac{1-r^{2}}{1-2r\cos(\theta-u)+r^{2}}d\nu(e^{i\theta}) \\
&= \nu(\{-1\})y+ \frac{1}{\pi}\int_{\mathbb{R}}\frac{y}{(x-t)^{2}+y^{2}}d\mu(t).
\end{align*}
We deduce that 
\begin{equation}\label{eq:5pmaqm6}
F(x+iy)= \alpha y  + \frac{1}{\pi}\int_{\mathbb{R}}\frac{y}{(x-t)^{2}+y^{2}}d\mu(t), ~~\forall~x+iy \in \mathbb{C}_{+}, 
\end{equation}
where  $\alpha= \nu(\{-1\}) \in  \mathbb{R}$.
Consider the sequence $(y_{n})_{n \in \mathbb{N}}$  of positive real numbers such that   $y_{ n}> 0$ for all $n\in \mathbb{N}$ and $\lim_{n\to \infty}y_{n}=+\infty$.  For $n\in \mathbb{N}$, we have
$$   |\alpha| \leq \frac{|F(iy_{n})|}{y_{n}} + \frac{1}{\pi}\int_{\mathbb{R}}\frac{1}{t^{2}+y_{n}^{2}}d|\mu|(t),     $$
thanks to Relation (\ref{eq:5pmaqm6}).  We deduce that
$$  
|F(iy_{n})| \leq \Phi^{-1}\left(\frac{2}{\pi y_{n}}\right)\|F\|_{h^{\Phi}}^{lux},
       $$
according to Proposition \ref{pro:main6mq4}. It follows that 
\begin{equation}\label{eq:5pmapaqmqm6}
\lim_{n \to \infty}\frac{|F(iy_{n})|}{y_{n}}=0.
\end{equation}
For $t\in \mathbb{R}$, put $$   f_{n}(t)= \frac{1}{t^{2}+y_{n}^{2}}d|\mu|(t).   $$
Since $ \int_{\mathbb{R}}\frac{d|\mu|(t)}{1+t^{2}} < \infty$ and  $$  f_{n}(t)= \frac{1+t^{2}}{t^{2}+y_{n}^{2}}\times\frac{d|\mu|(t)}{1+t^{2}},      $$
we deduce that $ \lim_{n \to \infty}f_{n}(t)=0$,  for almost all $t\in \mathbb{R}$ and 
 $$ |f_{n}(t)| \leq  \frac{d|\mu|(t)}{1+t^{2}},~~\forall~y_{ n} \geq  1,~~   \forall~t\in \mathbb{R}.     $$
It follows that 
\begin{equation}\label{eq:apmqm6}
\lim_{n\to \infty}\int_{\mathbb{R}}|f_{n}(t)|dt =0,
\end{equation}
thanks to dominated convergence theorem. From Relations (\ref{eq:5pmapaqmqm6}) and (\ref{eq:apmqm6}), we have $\alpha=0$. Therefore,
$$   F(x+iy)=  \frac{1}{\pi}\int_{\mathbb{R}}\frac{y}{(x-t)^{2}+y^{2}}d\mu(t), ~~\forall~x+iy \in \mathbb{C}_{+}. 
   $$
The Relation (\ref{eq:deltaqa2}) follows from \cite[Theorem 11.9]{javadmas}.
\epf

\subsection{Proof of Theorems \ref{pro:main11aqa6} and \ref{pro:main11pal6}.}

\begin{lem}\label{pro:mainfaqepq5}
Let $\Phi$ be a convex growth function such that $\Phi(t)>0$ for all $t>0$ and $F$  an analytic function on $\mathbb{C_{+}}$. The following assertions are equivalent.
\begin{itemize}
 \item[(i)] $ F\in H^{\Phi}(\mathbb{C_{+}})$.
 \item[(ii)] There exists a unique  $f \in L^{\Phi}(\mathbb{R})$ such that  $\log|f| \in L^{1}\left(\frac{dt}{1+t^{2}}\right)$ and
 $$  F(x+iy)= U_{f}(x+iy), ~~\forall~x+iy \in \mathbb{C}_{+}.    $$
\end{itemize}
Moreover, 
  $$  \|F\|_{H^{\Phi}}^{lux}=\lim_{y \to 0}\|F(.+iy)\|_{L^{\Phi}}^{lux}=\|f\|_{L^{\Phi}}^{lux}.
     $$
\end{lem}
 
 \begin{proof}
The  implication $(ii)\Rightarrow (i)$ is obvious.
 For the converse, for $\omega \in \mathbb{D}$, put
 $$ G(\omega)= F\left(i\frac{1-\omega}{1+\omega} \right).       $$
Since $G \in H^{1}(\mathbb{D})$,  there exists a unique function 
  $g\in L^{1}(\mathbb{T})$ such that $\log|g| \in L^{1}(\mathbb{T})$ and for all  $0 \leq r<1$,
 $$  G(\omega)= \frac{1}{2\pi}\int_{-\pi}^{\pi}\frac{1-r^{2}}{1-2r\cos(u-\theta)+r^{2}}g(e^{i\theta})d\theta.       $$
For  $z=x+iy\in \mathbb{C}_{+}$ and   $\omega=re^{iu} \in \mathbb{D}$  such that $ z=i\frac{1-\omega}{1+\omega}$, we have 
$$ F(x+iy)=G(re^{iu})= \frac{1}{2\pi}\int_{-\pi}^{\pi}\frac{1-r^{2}}{1-2r\cos(u-\theta)+r^{2}}g(e^{i\theta})d\theta = \frac{1}{\pi}\int_{\mathbb{R}}\frac{y}{(x-t)^{2}+y^{2}}g\circ\varphi^{-1}(t)dt.  $$
We deduce that  
 $$ F(x+iy)=U_{g\circ\varphi^{-1}}(x+iy), ~~\forall~x+iy \in \mathbb{C}_{+}.    $$ 
Since $g$ and $\log|g|$ belong to $L^{1}(\mathbb{T})$, we deduce that $g\circ \varphi^{-1}$ and $\log|g\circ \varphi^{-1}|$ belong to  $L^{1}\left(  \frac{dt}{1+t^{2}} \right)$. It follows that 
$g\circ\varphi^{-1}\in  L^{\Phi}\left(\mathbb{R}\right)$  and $$  \|F\|_{H^{\Phi}}^{lux}=\lim_{y\to 0}\|F(.+iy)\|_{L^{\Phi}}^{lux}=\|g\circ\varphi^{-1}\|_{L^{\Phi}}^{lux},     $$
according to Lemma \ref{pro:main 3aqqq2} and Corollary \ref{pro:main6apaqmq4}.
 \end{proof}

\begin{lem}\label{pro:main5aqq5}
Let $\beta >0$ and  $\Phi \in  \mathscr{U}$. If  $F\in H^{\Phi}(\mathbb{C_{+}})$ then we have
$$  \int_{\mathbb{R}}\frac{F(t+i\beta)}{t-\overline{z}}dt=0, ~~\forall~z \in \mathbb{C}_{+}.    $$
\end{lem}

\begin{proof}
Let us prove that the integral 
\begin{equation}\label{eq:5pmapmqm6}
\int_{\mathbb{R}}\frac{F(t+i\beta)}{t-\overline{z}}dt, ~~\forall~z \in \mathbb{C}_{+},
\end{equation}
is well defined.
For $z \in \mathbb{C}_{+}$, consider $\varphi_{\overline{z}}$ the function defined by
 $$   \varphi_{\overline{z}}(t)=\frac{1}{t-\overline{z}}, ~~\forall~t\in \mathbb{R}. 
    $$
Let's show that $\varphi_{\overline{z}} \in L^{\Psi}(\mathbb{R})$, where $\Psi$ is the complementary function of $\Phi$. 
Since $L^{\Phi}(\mathbb{R})$ is contained in $L^{a_\Phi}(\mathbb{R}) + L^{b_\Phi}(\mathbb{R})$, with  $1\leq a_\Phi \leq b_\Phi<\infty$. For $f\in L^{\Phi}(\mathbb{R})$,  there exists $f_{1} \in L^{a_\Phi}(\mathbb{R})$ and $f_{2} \in L^{b_\Phi}(\mathbb{R})$ such that $ f= f_{1}+f_{2}.$ We have
$$  \int_{\mathbb{R}}|\varphi_{\overline{z}}f(t)|dt \leq \int_{\mathbb{R}}|\varphi_{\overline{z}}f_{1}(t)|dt+\int_{\mathbb{R}}|\varphi_{\overline{z}}f_{2}(t)|dt \leq \|\varphi_{\overline{z}}\|_{q_{1}}\|f_{1}\|_{L^{a_\Phi}}+\|\varphi_{\overline{z}}\|_{q_{2}}\|f_{2}\|_{L^{b_\Phi}}
< \infty, 
      $$
where $\frac{1}{a_\Phi}+\frac{1}{q_{1}}=1$ and $\frac{1}{b_\Phi}+\frac{1}{q_{2}}=1$. We deduce that 
$$   \int_{\mathbb{R}}|\varphi_{\overline{z}}f(t)|dt < \infty, ~~\forall~f\in L^{\Phi}(\mathbb{R}).
     $$
It follows that $\varphi_{\overline{z}} \in L^{\Psi}(\mathbb{R})$, according to the inverse H\"older inequality in Orlicz spaces (see \cite[Proposition 1.1]{raoren}). Therefore, the integral(\ref{eq:5pmapmqm6}) is well defined. Indeed,  
$$   \left|\int_{\mathbb{R}}\frac{F(t+i\beta)}{t-\overline{z}}dt \right| \leq  \int_{\mathbb{R}}|F(t+i\beta)\varphi_{\overline{z}}(t)|dt \leq 2 \|\varphi_{\overline{z}}\|_{L^{\Psi}}^{lux}\|F(.+i\beta)\|_{L^{\Phi}}^{lux} \leq 2 \|\varphi_{\overline{z}}\|_{L^{\Psi}}^{lux}\|F\|_{H^{\Phi}}^{lux} < \infty. $$
Let $(O,\overrightarrow{u},\overrightarrow{v})$ be the canonical orthonormal frame of the complex plane. 
For $R > \max\{1, \beta+|z|\}$, consider $\Gamma_{R,\beta}$ the curve defined by 
 $$ \Gamma_{R,\beta}:=\{ \omega \in \mathbb{C}_{+}:\hspace*{0.25cm} |\omega-i\beta| \leq R \hspace*{0.25cm}\text{and} \hspace*{0.25cm} \mathrm{Im}(\omega) \geq \beta       \}.       $$
Since  $i\beta+ \overline{z}$ is not in the interior of  $\Gamma_{R,\beta}$, we deduce that
 $$  \int_{\Gamma_{R,\beta}}\frac{F(\zeta)}{\zeta-i\beta-\overline{z}}d\zeta=0,         $$
 according Cauchy’s integral formula.
Consider the  triangle $AOB$ rectangular at $O$ and contained in the frame  $(O,\overrightarrow{u},\overrightarrow{v})$ such that the distances $AO=\beta$ and $AB=R$, and put $\theta_{R}=mes(\widehat{OBA})$, the measure of the angle $OBA$. We have
$$  \int_{\theta_{R}}^{ \pi-\theta_{R}} \Phi^{-1}\left(\frac{2}{\pi R\sin \theta}\right)d\theta \lesssim   \int_{\theta_{R}}^{\pi/2} \Phi^{-1}\left(\frac{1}{ R\sin \theta}\right)d\theta  \lesssim  \int_{\theta_{R}}^{\pi/2} \Phi^{-1}\left(\frac{1}{ R\theta}\right)d\theta.      $$
Let $q\geq 1$ such that $\Phi \in \mathscr{U}^{q}$.  It follows that 
$$  \int_{\theta_{R}}^{\pi/2} \Phi^{-1}\left(\frac{1}{ R\theta}\right)d\theta  \lesssim  \int_{\theta_{R}}^{\pi/2} \Phi^{-1}\left(\frac{2}{ R\pi}\right)\times \frac{1}{\theta} d\theta \lesssim  \left(\frac{2}{ R \pi}\right)^{1/q}\times\ln\left(  \frac{\pi}{2\theta_{R}} \right),  $$
since    $\Phi^{-1}$ is  of lower-type $1/q$ and the function 
$t\mapsto \frac{\Phi^{-1}(t)}{t}$ is non-increasing on $\mathbb{R}_{+}^{*}$. We deduce that 
$$  \int_{\theta_{R}}^{ \pi-\theta_{R}} \Phi^{-1}\left(\frac{2}{\pi R\sin \theta}\right)d\theta \lesssim  \frac{\ln\left(  \frac{\pi R}{2\beta} \right)}{R^{1/q}},   $$
as $\theta_{R} \geq \sin \theta_{R}=\frac{\beta}{R}$. 
By the Proposition \ref{pro:main6mq4}, we have
\begin{align*}
\left|\int_{-R\cos \theta_{R}}^{R\cos \theta_{R}}\frac{F(t+i\beta)}{t-\overline{z}}dt \right|
&=\left|\int_{\pi-\theta_{R}}^{ \theta_{R}}\frac{F(Re^{i\theta})}{Re^{i\theta}-i\beta-\overline{z}}iRe^{i\theta}d\theta \right| \\
&=\left|-\int_{\theta_{R}}^{ \pi-\theta_{R}}\frac{F(R\cos \theta+iR\sin \theta)}{Re^{i\theta}-i\beta-\overline{z}}iRe^{i\theta}d\theta \right|\\
&\lesssim \int_{\theta_{R}}^{ \pi-\theta_{R}}\left|\frac{F(R\cos \theta+iR\sin \theta)}{Re^{i\theta}-i\beta-\overline{z}}iRe^{i\theta}\right|d\theta \\
&\lesssim \frac{R}{R- |i\beta+\overline{z}|}\|F\|_{H^{\Phi}}^{lux} \int_{\theta_{R}}^{ \pi-\theta_{R}} \Phi^{-1}\left(\frac{2}{\pi R\sin \theta}\right)d\theta. 
\end{align*}
We deduce that 
$$ \left|\int_{-R\cos \theta_{R}}^{R\cos \theta_{R}}\frac{F(t+i\beta)}{t-\overline{z}}dt \right| \lesssim \|F\|_{H^{\Phi}}^{lux} \times\frac{R}{R- |i\beta+\overline{z}|} \times \frac{\ln\left(  \frac{\pi R}{2\beta} \right)}{R^{1/q}} \longrightarrow 0, 
        $$
when $R\longrightarrow \infty$.         
\end{proof}

\proof[Proof of Theorem \ref{pro:main11aqa6}.]
The  implication $(ii)\Rightarrow (i)$ is obvious.
 For the converse, suppose that  $F\in H^{\Phi}(\mathbb{C}_{+})$. There exists a unique function $f\in L^{\Phi}\left(\mathbb{R}\right)$ such that
$$   F(x+iy)= U_{f}(x+iy), ~~\forall~x+iy \in \mathbb{C}_{+}       $$
and 
$$  \|F\|_{H^{\Phi}}^{lux}=\lim_{y \to 0}\|F(.+iy)\|_{L^{\Phi}}^{lux}=\|f\|_{L^{\Phi}}^{lux},
      $$
according to Lemma \ref{pro:mainfaqepq5}. Moreover,  
$$   \lim_{y \to 0}\| F(.+iy)-f\|_{L^{\Phi}}^{lux}=0,
     $$
according to Corollary \ref{pro:main 3v2qqakam}.
For $z \in \mathbb{C}_{+}$, consider $\varphi_{\overline{z}}$ the function defined above by
 $$   \varphi_{\overline{z}}(t)=\frac{1}{t-\overline{z}}, ~~\forall~t\in \mathbb{R}. 
    $$
Since   $\varphi_{\overline{z}} \in L^{\Psi}\left(\mathbb{R}\right)$ (see the proof of the Lemma \ref{pro:main5aqq5}), where $\Psi$ is the complementary function of $\Phi$,  we have for $\beta >0$,
\begin{align*}   
\left|\int_{\mathbb{R}}\frac{f(t)}{t-\overline{z}}dt-\int_{\mathbb{R}}\frac{F(t+i\beta)}{t-\overline{z}}dt \right|
&\leq  \int_{\mathbb{R}}|\varphi_{\overline{z}}(t)||F(t+i\beta)-f(t)|dt \\
&\leq 2\|\varphi_{\overline{z}} \|_{L^{\Psi}}^{lux}\| F(.+i\beta)-f\|_{L^{\Phi}}^{lux},
\end{align*}   
thanks to  H\"older's inequality in Orlicz spaces. We deduce that 
$$ \int_{\mathbb{R}}\frac{f(t)}{t-\overline{z}}dt=\lim_{\beta \to 0}\int_{\mathbb{R}}\frac{F(t+i\beta)}{t-\overline{z}}dt.         $$
It follows that  $f\in  H^{\Phi}\left(\mathbb{R}\right)$ since
$$   \int_{\mathbb{R}}\frac{f(t)}{t-\overline{z}}dt=0, ~~\forall~ z\in \mathbb{C}_{+},      $$
according to  Lemma \ref{pro:main5aqq5}.
For $z=x+iy\in \mathbb{C}_{+}$, we have 
  $$ \frac{y}{(x-t)^{2}+y^{2}} =\frac{i/2}{t-\overline{z}} - \frac{i/2}{t-z},              $$
and
$$ \frac{x-t}{(x-t)^{2}+y^{2}} =\frac{-1/2}{t-\overline{z}} + \frac{-1/2}{t-z}.              $$
We deduce that 
$$ F(z)=\frac{1}{2\pi i}\int_{\mathbb{R}}\frac{f(t)}{t-z}dt 
=\frac{i}{\pi}\int_{\mathbb{R}}\frac{x-t}{(x-t)^{2}+y^{2}}f(t)dt.
     $$
\epf

\proof[Proof of Theorem \ref{pro:main11pal6}.]
For  $z=x+iy \in \mathbb{C}_{+}$, we have   $$  \frac{1}{i\pi}\frac{1}{t-z} = \frac{1}{\pi}\frac{y}{(x-t)^{2}+y^{2}} + \frac{i}{\pi}\frac{x-t}{(x-t)^{2}+y^{2}},~~ \forall~t \in \mathbb{R}.        $$ 
It follows that
$$ S(f)(x+iy)
=U_{f}(x+iy) +iV_{f}(x+iy)= U_{f}(x+iy) +iU_{\mathcal{H}(f)}(x+iy)         $$
and 
$$  
\lim_{y\to 0}S(f)(x+iy)=f(x)+i\mathcal{H}(f)(x),    $$ 
for almost all $x\in \mathbb{R}$, according to Theorem \ref{pro:main11pqa6}. 
Since $U_{f}$ is a a real-valued harmonic function on $\mathbb{C}_{+}$,  $V_{f}$  its harmonic conjugate and   $\mathbb{C}_{+}$  a simply connected open set, we deduce that $S(f)$ is an analytic function on $\mathbb{C}_{+}$.  It follows that $S(f)\in  H^{\Phi}(\mathbb{C}_{+})$ if and only if
  $f+i\mathcal{H}(f) \in  H^{\Phi}(\mathbb{R})$, according to Theorem \ref{pro:main11aqa6}.  Therefore,
$$  \| f\|_{L^{\Phi}}^{lux}=\|U_{f}\|_{h^{\Phi}}^{lux}\leq\|S(f)\|_{H^{\Phi}}^{lux} \leq   \|U_{f}\|_{h^{\Phi}}^{lux}+\|V_{f}\|_{h^{\Phi}}^{lux} \lesssim  \| f\|_{L^{\Phi}}^{lux},   $$
thanks to  Lemma \ref{pro:main 3aqqq2} and  Theorem \ref{pro:main11pqa6}.
\epf

\subsection{Proof of Theorem \ref{pro:main 5dmw3pl}.}

\begin{prop}\label{pro:main 5apqaq8qllm}
Let $\Phi$ be an N$-$function and $f\in L^{\Phi}(\mathbb{R})$. If  $\Phi \in \nabla_{2}$ then  the following assertions are satisfied.
\begin{itemize}
\item[(i)] For almost all $x\in \mathbb{R}$, 
$$ \mathcal{H} \left( \mathcal{H}(f) \right)(x)=-f(x).      $$
\item[(ii)] There exist $c_{\Phi}$ and $C_{\Phi}$ constants depending only on $\Phi$ such that
\begin{equation}\label{eq:phimaaqde}
c_{\Phi}\| f\|_{L^{\Phi}}^{lux} \leq \| \mathcal{H}(f)\|_{L^{\Phi}}^{lux} \leq C_{\Phi}\| f\|_{L^{\Phi}}^{lux}.
\end{equation}
 \end{itemize}
\end{prop}

\begin{proof}
Without loss of generality, we assume that $f$ is real-valued. Since  $S(f)\in H^ {\Phi}(\mathbb{C}_{+})$ and 
\begin{equation}\label{eq:maaqqde}
\lim_{y\to 0}S(f)(x+iy)=f(x)+i\mathcal{H}(f)(x),
\end{equation}
for almost all $x\in \mathbb{R}$, according to Theorem \ref{pro:main11pal6}. It follows that
$$ S(f)(z)=\frac{1}{2\pi i}\int_{\mathbb{R}}\frac{f(t)+ i\mathcal{H}(f)(t)}{t-z}dt, ~~\forall~z\in  \mathbb{C_{+}},   $$
according to Theorem \ref{pro:main11aqa6}.
For $z=x+iy\in  \mathbb{C_{+}}$, we have 
$$  2S(f)(z)= \frac{1}{\pi i}\int_{\mathbb{R}}\frac{f(t)+ i\mathcal{H}(f)(t)}{t-z}dt =   S(f)(z)+\frac{1}{\pi }\int_{\mathbb{R}}\frac{\mathcal{H}(f)(t)}{t-z}dt.  $$
We deduce that 
$$ S(f)(z)=\frac{1}{\pi }\int_{\mathbb{R}}\frac{\mathcal{H}(f)(t)}{t-z}dt= -V_{\mathcal{H}(f)}(x+iy)+iU_{\mathcal{H}(f)}(x+iy)      $$
and 
\begin{equation}\label{eq:maqde}
\lim_{y\to 0}S(f)(x+iy)=-\mathcal{H}\left(\mathcal{H}(f)\right)(x)+i \mathcal{H}(f)(x),
\end{equation}
for almost all $x\in \mathbb{R}$. It follows that
 $$  \lim_{y\to 0}S(f)(x+iy)=f(x)+i\mathcal{H}(f)(x)=  -\mathcal{H}\left(\mathcal{H}(f)\right)(x)+i \mathcal{H}(f)(x),       $$
for almost all $x\in \mathbb{R}$, thanks to Relations (\ref{eq:maaqqde}) and (\ref{eq:maqde}). By identification, we get
 $$  \mathcal{H}\left(\mathcal{H}(f)\right)(x)=-f(x),     $$
for almost all $x\in \mathbb{R}$.  Since  $\mathcal{H}$ is bounded on $L^{\Phi}(\mathbb{R})$, we have 
$$  \|f\|_{L^{\Phi}}^{lux}=  \|\mathcal{H} \left( \mathcal{H}(f) \right)\|_{L^{\Phi}}^{lux} \lesssim \|\mathcal{H}(f)\|_{L^{\Phi}}^{lux}\lesssim \|f\|_{L^{\Phi}}^{lux}.
  $$
\end{proof}

The proof of Theorem \ref{pro:main 5dmw3pl} follows from Theorems \ref{pro:main11a6} and \ref{pro:main 5dmpmw3pl}, and  Proposition \ref{pro:main 5apqaq8qllm}. Therefore, it will not be written.

\subsection{Proof of Theorem \ref{pro:main 5aqk3pl}.}

\begin{prop}\label{pro:main90aqaqmq2}
Let $(\Phi,\Psi)$ be a complementary pair of $N-$functions. If both $\Phi$ and $\Psi$  satisfy to $\Delta_2-$condition then for   $(f,g) \in L^{\Phi}(\mathbb{R}) \times L^{\Psi}(\mathbb{R})$, we have
$$  \lim_{y\to 0}\int_{\mathbb{R}}U_{f}(x+iy)U_{g}(x+iy)dx =\int_{\mathbb{R}}f(x)g(x)dx, 
    $$
and 
$$  \lim_{y\to 0}\int_{\mathbb{R}}\left|U_{f}(x+iy)U_{g}(x+iy)\right|dx =\int_{\mathbb{R}}|f(x)g(x)|dx.
    $$
\end{prop}

\begin{proof}
The proof  follows from  Corollary \ref{pro:main 3v2qqakam} and H\"older's inequality in Orlicz spaces. 
\end{proof}

Let $(\Phi,\Psi)$ be a complementary pair of $N-$functions. The function
 $$  f\mapsto \|f\|_{L^{\Phi}}^{0}:= \sup\left\{  \int_{\mathbb{R}}|f(x)g(x)|dx    :  g\in L^{\Psi}(\mathbb{R})\hspace*{0.25cm} \text{with}\hspace*{0.25cm} \|g\|_{L^{\Psi}}^{lux} \leq 1 \right\},    $$
is a   norm on $L^{\Phi}(\mathbb{R})$. Moreover, 
 \begin{equation}\label{eq:noyaqaoson}
   \|f\|_{L^{\Phi}}^{lux} \leq \|f\|_{L^{\Phi}}^{0} \leq 2\|f\|_{L^{\Phi}}^{lux}, ~~ \forall~f \in L^{\Phi}(\mathbb{R}).
 \end{equation}

\begin{prop}\label{pro:mainaqqaqq2}
Let  $\Phi$ be an N$-$function and $f\in L^{\Phi}(\mathbb{R})$. If   $\Phi \in  \nabla_{2}$ then we have,
$$  \lim_{y\to 0}\|U_{f}(.+iy)-f\|_{L^{\Phi}}^{0}=0     $$
and 
$$ \lim_{y\to 0}\|U_{f}(.+iy)\|_{L^{\Phi}}^{0}= \|f\|_{L^{\Phi}}^{0}.   $$
\end{prop}

\begin{proof}
The proof  follows from  Corollary \ref{pro:main 3v2qqakam} and Relation (\ref{eq:noyaqaoson}).
\end{proof}

\begin{prop}\label{pro:mainqaqaqq2}
Let  $\Phi$ be an N$-$function such that  $\Phi \in  \nabla_{2}$. The map $F\mapsto \|F\|_{h^{\Phi}}^{0}$ is a norm on $h^{\Phi}(\mathbb{C}_{+})$. Moreover, 
\begin{equation}\label{eq:noyaonqaoson}
   \|F\|_{h^{\Phi}}^{lux} \leq \|F\|_{h^{\Phi}}^{0} \leq 2\|F\|_{h^{\Phi}}^{lux}, ~~ \forall~F\in  h^{\Phi}(\mathbb{C}_{+}).
 \end{equation}
\end{prop}

\begin{proof}
It follows from the H\"older inequality in Orlicz spaces that the map $F\mapsto \|F\|_{h^{\Phi}}^{0}$ is well defined on $h^{\Phi}(\mathbb{C}_{+})$. The homogeneity and the triangle inequality are obvious. It remains to prove that for $F\in h^{\Phi}(\mathbb{C}_{+})$, if $\|F\|_{h^{\Phi}}^{0}=0$ then $F\equiv 0$.
Let  $F\in h^{\Phi}(\mathbb{C}_{+})$. 
Suppose that $F\not\equiv 0$. There exists a unique function
 $f \in L^{\Phi}(\mathbb{R})$ such that  
$$  F(x+iy)= U_{f}(x+iy), ~~ \forall~ x+iy \in \mathbb{C_{+}}    $$
and $ \|F\|_{h^{\Phi}}^{lux}= \| f\|_{L^{\Phi}}^{lux}$,  according to Theorem \ref{pro:main11a6}.   Put   
  $$ A:= \{x\in \mathbb{R} : | f(x)|>0 \}.     $$
Since $|A| \not=0$,  we can find a measurable subset $K$ of $\mathbb{R}$ such that $K \subset A$ and $0< | K| < \infty$. Consider the function $g$ defined by
$$  g(x)=\Psi^{-1}\left(\frac{1}{|K|} \right)\chi_{K}(x), ~~ \forall~ x\in \mathbb{R},        $$
where  $\Psi$ is the complementary function of $\Phi$. By construction,  $g \in L^{\Psi}(\mathbb{R})$ and $\| g\|_{L^{\Psi}}^{lux} \leq 1$. Indeed,
$$ \int_{\mathbb{R}} \Psi\left(| g(x)| \right)dx = \int_{\mathbb{R}} \Psi\left(\Psi^{-1}\left(\frac{1}{|K|} \right)\chi_{K}(x) \right)dx =\frac{1}{|K|}\int_{K} \Psi\left(k \right)dx =1.  $$
 It follows that,  $U_{g}\in h^{\Psi}(\mathbb{C}_{+})$ and  $\|U_{g}\|_{h^{\Psi}}^{lux} = \|g\|_{L^{\Psi}}^{lux}$.  We deduce that  
$$ \|F\|_{h^{\Phi}}^{0}\geq    \int_{\mathbb{R}}\left|f(x)\Psi^{-1}\left(\frac{1}{|K|} \right)\chi_{K}(x)\right|dx=\Psi^{-1}\left(\frac{1}{|K|} \right) \int_{K}|f(x)|dx >0,
     $$
thanks to Proposition \ref{pro:main90aqaqmq2}. Therefore, $\|F\|_{h^{\Phi}}^{0}\not=0$.\\
Let us now prove Relation (\ref{eq:noyaonqaoson}).
For $F \in h^{\Phi}(\mathbb{C}_{+})$,  there exists a unique function
 $f \in L^{\Phi}(\mathbb{R})$ such that 
$$  F(x+iy)= U_{f}(x+iy), ~~ \forall~ x+iy \in \mathbb{C_{+}}    $$
and  $ \|F\|_{h^{\Phi}}^{lux}= \| f\|_{L^{\Phi}}^{lux}$.   It will suffice to show that
\begin{equation}\label{eq:noyaonqaaqmmaoson}
\| F\|_{h^{\Phi}}^{0}=  \| f\|_{L^{\Phi}}^{0},
 \end{equation}
in view of Relation (\ref{eq:noyaqaoson}).
Let $G \in h^{\Psi}(\mathbb{C}_{+})$ such that $\|G\|_{h^{\Psi}}^{lux} \leq 1$. There exists a unique function
 $g \in L^{\Psi}(\mathbb{R})$ such that
$$  G(x+iy)= U_{g}(x+iy), ~~ \forall~ x+iy \in \mathbb{C_{+}}     $$
and  $ \|G\|_{h^{\Psi}}^{lux}= \| g\|_{L^{\Psi}}^{lux}$. We deduce that
$$  \lim_{y\to 0}\int_{\mathbb{R}}\left|F(x+iy)G(x+iy)\right|dx =\int_{\mathbb{R}}|f(x)g(x)|dx \leq  \| f\|_{L^{\Phi}}^{0},    $$
according to the Proposition \ref{pro:main90aqaqmq2}. It follows that
\begin{equation}\label{eq:noyaonqamaosaqppon}
    \|F\|_{h^{\Phi}}^{0} \leq \| f\|_{L^{\Phi}}^{0}.
 \end{equation}
Let  $g \in L^{\Psi}(\mathbb{R})$ such that $\| g\|_{L^{\Psi}}^{lux}\leq 1$. Since  $U_{g}\in h^{\Psi}(\mathbb{C}_{+})$ and  $\|U_{g}\|_{h^{\Psi}}^{lux} = \|g\|_{L^{\Psi}}^{lux}$, and  as 
$$   \int_{\mathbb{R}}|f(x)g(x)|dx=  \lim_{y\to 0}\int_{\mathbb{R}}\left|U_{f}(x+iy)U_{g}(x+iy)\right|dx \leq  \| F\|_{h^{\Phi}}^{0},     $$
we deduce that
\begin{equation}\label{eq:noyaonqamaoson}
\| f\|_{L^{\Phi}}^{0}  \leq  \|F\|_{h^{\Phi}}^{0} .
 \end{equation}
The conclusion follows from the Relations (\ref{eq:noyaonqamaosaqppon}) and (\ref{eq:noyaonqamaoson}).
\end{proof}

\begin{lem}\label{pro:mainqaqqaq2}
Let $\Phi$ be an N$-$function such that $\Phi \in \nabla_{2}$.  For all  $F\in h^{\Phi}(\mathbb{C}_{+})$  there exists a unique $f\in  L^{\Phi}(\mathbb{R})$ such that
$$  F(x+iy)= U_{f}(x+iy), ~~ \forall~ x+iy \in \mathbb{C_{+}}.     $$
Moreover, 
$$  \lim_{y\to 0}\|F(.+iy)-f\|_{L^{\Phi}}^{0}=0     $$
and 
$$\|F\|_{h^{\Phi}}^{0}=\lim_{y\to 0}\|U_{f}(.+iy)\|_{L^{\Phi}}^{0}=\|f\|_{L^{\Phi}}^{0}.$$
\end{lem}

\begin{proof}
The proof  follows from Theorem \ref{pro:main11a6}, Proposition \ref{pro:mainaqqaqq2} and Relation (\ref{eq:noyaonqaaqmmaoson})
\end{proof}

\begin{cor}\label{pro:mainqpmqapm2}
 Let $(\Phi,\Psi)$ be a complementary pair of $N-$functions such that both $\Phi$ and $\Psi$  satisfy to $\Delta_2-$condition. For  $G\in   h^{\Psi}(\mathbb{C}_{+})$, the map  $T_{G}$ defined by
 $$   T_{G}(F)=\lim_{y\to 0}\int_{\mathbb{R}}F(x+iy)G(x+iy)dx, ~~ \forall~F\in h^{\Phi}(\mathbb{C}_{+}),    $$
is a continuous linear form on   $h^{\Phi}(\mathbb{C}_{+})$. Moreover, 
 $$   \|T_{G}\|_{(h^{\Phi})^{*}} =  \|G\|_{h^{\Psi}}^{0}.  $$
  \end{cor}

\begin{proof}
The proof follows from Lemma \ref{pro:mainqaqqaq2}.
\end{proof}

\begin{lem}\label{pro:main 53ppaql}
Let $(\Phi,\Psi)$ be a complementary pair of $N-$functions such that both $\Phi$ and $\Psi$  satisfy to $\Delta_2-$condition. For  $f\in L^{\Phi}(\mathbb{R})$ and  $g\in L^{\Psi}(\mathbb{R})$, we have
\begin{equation}\label{eq:noyaonqavmaoson}
\int_{\mathbb{R}}\mathcal{H}(f)(x)\mathcal{H}(g)(x)dx =\int_{\mathbb{R}}f(x)g(x)dx
 \end{equation}
and 
\begin{equation}\label{eq:noyaonqamakoson}
\int_{\mathbb{R}}\mathcal{H}(f)(x)g(x)dx =-\int_{\mathbb{R}}f(x)\mathcal{H}(g)(x)dx.
 \end{equation}
\end{lem}

 \begin{proof}
Let  $f\in L^{\Phi}(\mathbb{R})$ and $g\in L^{\Psi}(\mathbb{R})$. We can assume without loss of generality that $f$ and $g$ are real-valued. Since  $S(f) \in H^{\Phi}(\mathbb{C}_{+})$ and $S(g) \in H^{\Psi}(\mathbb{C}_{+} )$, we deduce that the product $S(f)S(g)$ belongs to $ H^{1}(\mathbb{C}_{+})$, according H\"older's inequality in Orlicz spaces. It follows that
$$ p(x):=\lim_{y\to 0}S(f)(x+iy)S(g)(x+iy)= \left( f(x)+i\mathcal{H}(f)(x)\right) \left( g(x)+i\mathcal{H}(g)(x)\right),         $$    
for almost any $x\in \mathbb{R}$ and 
 $$ \int_{\mathbb{R}}p(x)dx=0,              $$
(see.  \cite[corollary 13.7]{javadmas}).  Since 
 $$ p= \left(fg-\mathcal{H}(f)\mathcal{H}(g)\right) +i \left( \mathcal{H}(f)g + \mathcal{H}(g)f \right), $$
the equalities (\ref{eq:noyaonqavmaoson}) and (\ref{eq:noyaonqamakoson}) follow.
\end{proof}

\proof[Proof of Theorem \ref{pro:main 5aqk3pl}.]
$i)$ For $G \in h^{\Phi}(\mathbb{C}_{+})$, consider $T_{G}$ the map   defined by
 $$   T_{G}(F)=\lim_{y\to 0}\int_{\mathbb{R}}F(x+iy)G(x+iy)dx, ~~ \forall~F\in h^{\Phi}(\mathbb{C}_{+}).    $$ 
The map $G\longmapsto T_{G}$ is continuous, linear and isometric from $h^{\Phi}(\mathbb{C}_{+})$ to $\left(h^{\Phi}(\mathbb{C}_{+})\right)^{*}$, according to Corollary \ref{pro:mainqpmqapm2}. We still have to prove that $G\longmapsto T_{G}$ is surjective.
Let   $T\in \left(h^{\Phi}(\mathbb{C}_{+})\right)^{*}$ and put for $f\in L^{\Phi}(\mathbb{R})$,   
 $$ T_{1}(f)= T(U_{f}).   $$
By construction, $T_{1}\in \left(L^{\Phi}(\mathbb{R})\right)^{*}$. Indeed,  
$$ | T_{1}(f) | = | T(U_{f})| \leq \|T\|_{\left(h^{\Phi}\right)^{*}} \|U_{f}\|_{h^{\Phi}}^{lux} \leq \|T\|_{\left(h^{\Phi}\right)^{*}}\|f\|_{L^{\Phi}}^{lux}.  $$
So there is a unique function $ g\in L^{\Psi}(\mathbb{R})$ such that $$ T_{1}(f)=\int_{\mathbb{R}}f(x)g(x)dx, ~~ \forall~f\in L^{\Phi}(\mathbb{R}).   $$
For  $F\in h^{\Phi}(\mathbb{C}_{+})$, there exists a unique function $f \in L^{\Phi}(\mathbb{R})$ such that  $$   F(x+iy)=U_{f}(x+iy), ~~ \forall~ x+iy \in \mathbb{C}_{+}.  $$
Since $G:=U_{g}\in h^{\Psi}(\mathbb{C}_{+})$, it follows that 
$$  \int_{\mathbb{R}}f(x)g(x)dx
 = \lim_{y\to 0}\int_{\mathbb{R}}F(x+iy)G(x+iy)dx,    $$
according to Proposition \ref{pro:main90aqaqmq2}. We deduce that 
 $$ T(F)= \lim_{y\to 0}\int_{\mathbb{R}}F(x+iy)G(x+iy)dx, ~~ \forall~F\in h^{\Phi}(\mathbb{C}_{+}).
        $$
Therefore, $G\longmapsto T_{G}$ is surjective from $h^{\Phi}(\mathbb{C}_{+})$ to $\left(h^{\Phi}(\mathbb{C}_{+})\right)^{*}$.\\ 
\\$ii)$ For $G \in H^{\Phi}(\mathbb{C}_{+})$, consider $T_{G}$ the map   defined by
 $$   T_{G}(F)=\lim_{y\to 0}\int_{\mathbb{R}}F(x+iy)\overline{G(x+iy)}dx, ~~ \forall~F\in H^{\Phi}(\mathbb{C}_{+}).    $$
By construction, the anti-linearity, the injectivity and the continuity of the map $G\longmapsto T_{G}$ on $H^{\Phi}(\mathbb{C}_{+})$ are obvious. We still have to prove that $G\longmapsto T_{G}$ is surjective.
For  $L\in \left(H^{\Phi}(\mathbb{C}_{+})\right)^{*}$, consider $L_{\mathbb{R}}$, the map defined by
$$  L_{\mathbb{R}}(F)= \mathrm{Re}\left( L(F)\right),  ~~ \forall~F\in H^{\Phi}(\mathbb{C}_{+}).        $$
By construction, $L_{\mathbb{R}}$ is a real-valued continuous linear form on $H^{\Phi}(\mathbb{C}_{+})$.
Moreover,
$$  L(F) =L_{\mathbb{R}}(F) - iL_{\mathbb{R}}(iF), ~~ \forall~F\in  H^{\Phi}(\mathbb{C}_{+}).   $$ 
Consider $L_{1}$ the map  defined by $$ L_{1}(f) = L_{\mathbb{R}}\left(U_{f}+iU_{\mathcal{H}(f)}\right) , ~~ \forall~f\in L_{\mathbb{R}}^{\Phi},      $$
where  $L_{\mathbb{R}}^{\Phi}$ is the subset of $L^{\Phi}(\mathbb{R})$ consisting of real-valued functions. By construction,  $L_{1}$ is a continuous linear form on  $L_{\mathbb{R}}^{\Phi}$. Indeed,
$$ | L_{1}(f)| = \left| L_{\mathbb{R}}\left(U_{f}+iU_{\mathcal{H}(f)}\right)\right| \lesssim \|L\|_{\left(H^{\Phi}\right)^{*}} \|U_{f}+iU_{\mathcal{H}(f)}\|_{H^{\Phi}}^{lux}\lesssim \|L\|_{\left(H^{\Phi}\right)^{*}} \|f\|_{L^{\Phi}}^{lux}, ~~ \forall~f\in L_{\mathbb{R}}^{\Phi}.  $$
There exists $\widetilde{L_{1}}\in \left(L^{\Phi}(\mathbb{R})\right)^{*}$ such that
$$  \widetilde{L_{1}}(f)=L_{1}(f),  ~~ \forall~f\in L_{\mathbb{R}}^{\Phi},      $$
thanks to  Hahn's Banach theorem.  
There thus exists a unique $g\in L^{\Psi}(\mathbb{R})$ with real values such that
$$  \widetilde{L_{1}}(f)=\int_{\mathbb{R}}f(x)g(x)dx, ~~ \forall~f\in L_{\mathbb{R}}^{\Phi}.      $$
We deduce that 
\begin{equation}\label{eq:noyaonqaplson}
L_{\mathbb{R}}\left( U_{f} +iU_{\mathcal{H}(f)}\right)= \int_{\mathbb{R}}f(x)g(x)dx, ~~ \forall~f\in L_{\mathbb{R}}^{\Phi}.
 \end{equation}
For  $f\in L_{\mathbb{R}}^{\Phi}$, we have
$$   L_{\mathbb{R}}\left(i( U_{f} +iU_{\mathcal{H}(f)})\right)=  L_{\mathbb{R}}\left(U_{\mathcal{H}(-f)} +iU_{\mathcal{H}(\mathcal{H}(-f))}\right) = \int_{\mathbb{R}}\mathcal{H}(-f)(x)g(x)dx,    $$
 since $\mathcal{H}(\mathcal{H}(f))=-f$, almost everywhere on  $\mathbb{R}$. We deduce that 
\begin{equation}\label{eq:noyaonqaplsghon}
 L_{\mathbb{R}}\left(i( U_{f} +iU_{\mathcal{H}(f)})\right) =-\int_{\mathbb{R}}\mathcal{H}(f)(x)g(x)dx, ~~ \forall~f\in L_{\mathbb{R}}^{\Phi}.
 \end{equation}
For $F\in H^{\Phi}(\mathbb{C}_{+})$, there is a unique function  $f \in H^{\Phi}(\mathbb{R})$ such that
$$  F(x+iy)= \frac{1}{2\pi i}\int_{\mathbb{R}}\frac{f(t)}{t-z}dt=\frac{1}{2}\left(U_{f}(x+iy)+iU_{\mathcal{H}(f)}(x+iy) \right),  ~~\forall~z=x+iy \in \mathbb{C}_{+},      $$
according to  Theorem \ref{pro:main11aqa6}. Let $f=f_{1}+if_{2}$,
where $f_{1}$ and $f_{2}$ are real-valued functions,
we have
\begin{align*}
L(F)
&= \frac{1}{2}\left(L_{\mathbb{R}}\left( U_{f_{1}-\mathcal{H}(f_{2})} +i U_{\mathcal{H}(f_{1})-\mathcal{H}(\mathcal{H}(f_{2}))}\right)-iL_{\mathbb{R}}\left(i\left( U_{f_{1}-\mathcal{H}(f_{2})} +i U_{\mathcal{H}(f_{1})-\mathcal{H}(\mathcal{H}(f_{2}))}\right)  \right)  \right) \\
&=\frac{1}{2} \left( \int_{\mathbb{R}}\left(f_{1}(x) -\mathcal{H}(f_{2})(x)  \right)g(x)dx +i \int_{\mathbb{R}}\left( \mathcal{H}(f_{1})(x)-\mathcal{H}(\mathcal{H}(f_{2}))(x) \right) g(x)dx      \right) \\
&=\frac{1}{2}\left( \int_{\mathbb{R}}f(x)g(x)dx + i\int_{\mathbb{R}}\mathcal{H}(f)(x)g(x)dx    \right),
\end{align*}
thanks to Relations (\ref{eq:noyaonqaplson}) and  (\ref{eq:noyaonqaplsghon}). We deduce that
$$ L(F) =\frac{1}{4}\int_{\mathbb{R}}\left(f(x)+i \mathcal{H}(f)(x)\right)\left(g(x)-i \mathcal{H}(g)(x)\right)dx,    $$
according to  Lemma \ref{pro:main 53ppaql}.  Taking   $G=U_{g/4}+iU_{\mathcal{H}(g/4)} \in H^{\Psi}(\mathbb{C}_{+})$, it follows that
$$ L(F) = \lim_{y\to 0}\int_{\mathbb{R}}F(x+iy)\overline{G(x+iy)}dx.     $$
\epf

\bibliographystyle{plain}

\end{document}